\documentclass[11pt]{article}%
\usepackage{amsfonts}
\usepackage{amsmath}
\usepackage{amssymb}
\usepackage{graphicx}%
\providecommand{\U}[1]{\protect\rule{.1in}{.1in}}
\newtheorem{theorem}{Theorem}

\newtheorem{corollary}[theorem]{Corollary}

\newtheorem{definition}[theorem]{Definition}

\newtheorem{lemma}[theorem]{Lemma}

\newtheorem{proposition}[theorem]{Proposition}
\newtheorem{remark}[theorem]{Remark}

\newenvironment{proof}[1][Proof]{\noindent\textbf{#1.} }{\ \rule{0.5em}{0.5em}}
\begin{document}

\title{A symplectic extension map and a new Shubin class of pseudo-differential operators}

\author{Nuno Costa Dias\textbf{\thanks{Corresponding author; ncdias@meo.pt}}
\and Maurice A. de
Gosson\textbf{\thanks{maurice.de.gosson@univie.ac.at}} \and
Jo\~{a}o Nuno Prata\textbf{\thanks{joao.prata@mail.telepac.pt }}}
\maketitle

\begin{abstract}
For an arbitrary pseudo-differential operator $A:\mathcal{S}(\mathbb{R}%
^{n})\longrightarrow\mathcal{S}^{\prime}(\mathbb{R}^{n})$ with Weyl symbol
$a\in\mathcal{S}^{\prime}(\mathbb{R}^{2n})$, we consider the
pseudo-differential operators $\widetilde{A}:\mathcal{S}(\mathbb{R}%
^{n+k})\longrightarrow\mathcal{S}^{\prime}(\mathbb{R}^{n+k})$
associated with the Weyl symbols
$\widetilde{a}=(a\otimes1_{2k})\circ{s}$, where $1_{2k}(x)=1$ for
all $x\in\mathbb{R}^{2k}$ and ${s}$ is a linear symplectomorphism
of $\mathbb{R}^{2(n+k)}$. We call the operators $\widetilde{A}$
symplectic dimensional extensions of $A$. In this paper we study
the relation between $A$ and $\widetilde{A}$ in detail, in
particular their regularity, invertibility and spectral
properties. We obtain an explicit formula allowing to express the
eigenfunctions of $\widetilde{A}$ in terms of those of $A$. We use
this formalism to construct new classes of pseudo-differential
operators, which are extensions of the Shubin classes
$HG_{\rho}^{m_{1},m_{0}}$ of globally hypoelliptic operators. We
show that the operators in the new classes share the invertibility
and spectral properties of the operators in $HG_{\rho
}^{m_{1},m_{0}}$ but not the global hypoellipticity property.
Finally, we study a few examples of operators that belong to the
new classes and which are important in mathematical physics.
\end{abstract}


{\vskip 0.3cm}


\section{Introduction}

Every continuous linear operator $A:\mathcal{S}(\mathbb{R}^{n})\longrightarrow
\mathcal{S}^{\prime}(\mathbb{R}^{n})$ can be written as a pseudo-differential
operator
\begin{equation}
A\psi(x)=\frac{1}{(2\pi)^{n}}\int_{\mathbb{R}^{n}\times\mathbb{R}^{n}%
}e^{i(x-y)\cdot\xi}a_{\tau}((1-\tau)x+\tau y,\xi)\psi(y)\,dyd\xi
\label{Tau-PDO}%
\end{equation}
in terms of its $\tau$-symbol $a_{\tau}\in\mathcal{S}^{\prime}(\mathbb{R}%
^{n}\times\mathbb{R}^{n})$ (Shubin \cite{sh87}); the integral is convergent
for $a_{\tau}\in\mathcal{S}(\mathbb{R}^{n}\times\mathbb{R}^{n})$ and $\psi
\in\mathcal{S}(\mathbb{R}^{n})$, and should otherwise be interpreted in the
distributional sense. The usual left, right and Weyl pseudo-differential
operators correspond to the cases $\tau=0,1$ and $1/2$, respectively.

Let $\mathcal{L}(\mathcal{S}(\mathbb{R}^{n}),\mathcal{S}^{\prime}%
(\mathbb{R}^{n}))$ be the space of linear and continuous operators
$A:\mathcal{S}(\mathbb{R}^{n})\longrightarrow\mathcal{S}^{\prime}%
(\mathbb{R}^{n})$. In this paper we will define and study a family of
embedding maps ${\mathbb{E}}_{s}:\mathcal{L}(\mathcal{S}(\mathbb{R}%
^{n}),\mathcal{S}^{\prime}(\mathbb{R}^{n}))\longrightarrow\mathcal{L}%
(\mathcal{S}(\mathbb{R}^{n+k}),\mathcal{S}^{\prime}(\mathbb{R}^{n+k}))$
indexed by the elements $s$ of the symplectic group of
$\mathbb{R}^{2(n+k)}$, which we will call \textit{symplectic
dimensional extension maps}. In a nutshell, we are interested in
these transformations because (i) they generate large classes of
important operators $\widetilde{A}\in\mathcal{L}(\mathcal{S}(\mathbb{R}%
^{n+k}),\mathcal{S}^{\prime}(\mathbb{R}^{n+k}))$ and (ii) many properties of
the operators ${\mathbb{E}}_{s}[A]$, including the spectral properties, can be
completely determined from those of $A$.

From now on we will write the Weyl symbols simply as $a=a_{1/2}$ and denote
the Weyl correspondence by $A\overset{\text{Weyl}}{\longleftrightarrow}a$ or
$a\overset{\text{Weyl}}{\longleftrightarrow}A$. The spaces $\mathbb{R}^{2n}$
and $\mathbb{R}^{2k}$ are equipped with the standard symplectic forms denoted
by $\sigma_{n}$ and $\sigma_{k}$, respectively. In $\mathbb{R}^{2(n+k)}$ we
have the symplectic form $\sigma_{n+k}=\sigma_{n}\oplus\sigma_{k}$.

Since the operators $A\in\mathcal{L}(\mathcal{S}(\mathbb{R}^{n}),\mathcal{S}%
^{\prime}(\mathbb{R}^{n}))$ are in one-to-one correspondence with their Weyl
symbols $a\in\mathcal{S}^{\prime}(\mathbb{R}^{2n})$, an arbitrary map for
operators can be defined in terms of the corresponding map for symbols. We
then present our main definition

\begin{definition}
\label{Extmap} For arbitrary $k \in\mathbb{N}_{0}$ define the embedding map
\begin{equation}
{E}_{s} :\mathcal{S}^{\prime}(\mathbb{R}^{2n})\longrightarrow\mathcal{S}%
^{\prime}(\mathbb{R}^{2(n+k)});\quad a \longmapsto\widetilde{a}={E}_{s}
[a]=(a\otimes1_{2k})\circ{s} \label{ext3}%
\end{equation}
where $1_{{2k}}:\mathbb{R}^{2k}\longrightarrow\mathbb{R}$ is the
trivial function $1_{{2k}}(y,\eta)=1$ and ${s}$ is a linear
symplectomorphism in $\mathbb{R}^{2(n+k)}$. A symplectic
dimensional extension map $\mathbb{E}_{s}$ is an embedding map
\begin{gather}
{\mathbb{E}}_{s}:\mathcal{L}(\mathcal{S}(\mathbb{R}^{n}),\mathcal{S}^{\prime
}(\mathbb{R}^{n}))\longrightarrow\mathcal{L}(\mathcal{S}(\mathbb{R}%
^{n+k}),\mathcal{S}^{\prime}(\mathbb{R}^{n+k}))\label{ext3bis}\\
\,A\longmapsto\widetilde{A}_{}={\mathbb{E}}_{ {s}}[A]\nonumber
\end{gather}
uniquely defined, for each $E_{s}$, by the following commutative diagram
\begin{equation}%
\begin{array}
[c]{ccccc}
& a & \overset{\text{Weyl}}{<--------->} & A & \\
& | &  & | & \\
{E}_{s} & | &  & | & {\mathbb{E}}_{s}\\
& \downarrow &  & \downarrow & \\
& \widetilde{a} & \overset{{\text{Weyl}}}{<--------->} & \widetilde{A} &
\end{array}
\label{diag}%
\end{equation}

\end{definition}

\begin{remark}
In order to write the symbol $\widetilde{a}$ (\ref{ext3}) more explicitly let
us make the identifications
\[
\mathbb{R}^{2n}=\mathbb{R}_{x}^{n}\times\mathbb{R}_{\xi}^{n} \quad,
\quad\mathbb{R}^{2k} =\mathbb{R}_{y}^{k}\times\mathbb{R}_{\eta}^{k}
\]
\[
\mathbb{R}^{2(n+k)} =\mathbb{R}_{x,y}^{(n+k)}\times\mathbb{R}_{\xi,\eta
}^{(n+k)}
\]
Then for $s=I$ we have in coordinates $\widetilde{a}(x,y;\xi,\eta)=a(x,\xi)$.
In the general case
\begin{equation}
s: \mathbb{R}^{2(n+k)} \longrightarrow\mathbb{R}^{2(n+k)}; (x,y;\xi,\eta)
\longrightarrow(x^{\prime},y^{\prime};\xi^{\prime},\eta^{\prime}%
)=s(x,y;\xi,\eta)
\end{equation}
and
\begin{equation}
\widetilde{a}(x,y;\xi,\eta)=a(x^{\prime}(x,y;\xi,\eta),\xi^{\prime}%
(x,y;\xi,\eta)).
\end{equation}

\end{remark}

The class of operators of the form $\widetilde{A}=\mathbb{E}_{s}[A]$ is quite
large. A few examples are:

\begin{itemize}
\item All partial differential operators $\widetilde{A}=\partial_{
x_{i}}$ on $\mathbb{R}^{m}$ are symplectic dimensional extensions
of the ordinary derivative operator $A=\partial_{x}$ on
$\mathbb{R}$. In this case $n=1$, $k=m-1$, ${s}=I$ and
$a(x,\xi)=i\xi$.

\item The Landau Hamiltonian \cite{lali97}
\begin{equation}
\widetilde{H}_{L}=-(\partial_{x}^{2}+\partial_{y}^{2})+i(x\partial
_{y}-y\partial_{x})+\tfrac{1}{4}(x^{2}+y^{2}) \label{LandauH}%
\end{equation}
which describes the motion of a test particle in the presence of a
magnetic field, is a symplectic dimensional extension of the
harmonic oscillator Hamiltonian $H_{0}= -\partial_{x}^{2}+x^{2}$
(section \ref{secLandau}).

\item The Bopp pseudo-differential operators $\widetilde{A}_{B}:\mathcal{S}%
(\mathbb{R}^{2n})\longrightarrow\mathcal{S}^{\prime}(\mathbb{R}^{2n})$,
formally $\widetilde{A}_{B}=a\star$ where $a\in\mathcal{S}^{\prime}%
(\mathbb{R}^{2n})$ and $\star$ is the Moyal star product, are very
important in the deformation quantization of Bayen et al.
\cite{BFFLS1,BFFLS2,Bopp1}. The operators
$\widetilde{A}_{B}=a\star$ are symplectic dimensional extensions
of the Weyl operators $A\overset{\text{Weyl}}{\longleftrightarrow}
a$ (section \ref{secBopp}).
\end{itemize}

The first part of this paper is devoted to study the relation between the
operators $A \in\mathcal{L}(\mathcal{S}(\mathbb{R}^{n}),\mathcal{S}^{\prime
}(\mathbb{R}^{n}))$ and their dimensional extensions $\widetilde{A}%
=\mathbb{E}_{s}[A]$. We prove several general results about the regularity,
invertibility and spectral properties of $\widetilde{A}$. In particular, we
show that, in the general case, the eigenfunctions of $\widetilde{A}$ can be
completely determined from those of $A$.

In the second part of the paper we define new classes of pseudo-differential
operators, which are extensions of the Shubin classes of globally hypoelliptic
operators. Using the results of the first part, we prove a complete set of
results about the spectral, invertibility and hypoellipticity properties of
the operators in the new classes. Finally, in section \ref{secExa} we present
several examples of operators that belong to the new classes and have
important applications in quantum mechanics and deformation
quantization.\bigskip

\noindent\textbf{Notation}. We will denote by $\mathcal{S}(\mathbb{R}^{m})$
the Schwartz space of rapidly decreasing functions on $\mathbb{R}^{m}$; its
dual $\mathcal{S}^{\prime}(\mathbb{R}^{m})$ is the space of tempered
distributions. The space of linear and continuous operators of the form
$\mathcal{S}(\mathbb{R}^{m})\longrightarrow\mathcal{S}^{\prime}(\mathbb{R}%
^{m})$ is denoted by $\mathcal{L}(\mathcal{S}(\mathbb{R}^{m}),\mathcal{S}%
^{\prime}(\mathbb{R}^{m}))$. The scalar product of two functions $f,g\in
L^{2}(\mathbb{R}^{m})$ is denoted by $(f|g)$ and the corresponding norm by
$||\cdot||$. The distributional bracket is $\langle\,,\,\rangle$. The
Euclidean product of two vectors $x,y$ in $\mathbb{R}^{m}$ is written $x\cdot
y$ and the norm of $x\in\mathbb{R}^{m}$ is $|x|$. The Weyl correspondence is
denoted by $A\overset{\text{Weyl}}{\longleftrightarrow}a$ or
$a\overset{\text{Weyl}}{\longleftrightarrow}A$. The Weyl operators are also
written $A=$ $Op^{w}(a)$. The symplectic form on $\mathbb{R}^{2n}$ is defined
by $\sigma_{n}(z,z^{\prime})=Jz\cdot z^{\prime}$ where $J=%
\begin{pmatrix}
0 & I\\
-I & 0
\end{pmatrix}
$. For $z=(x,\xi)$ and $z^{\prime}=(x^{\prime},\xi^{\prime})$ we have
explicitly $\sigma_{n}(z,z^{\prime})=\xi\cdot x^{\prime}-\xi^{\prime}\cdot x$.
The extension to higher dimensions is obviously $\sigma_{n+k}=\sigma_{n}%
\oplus\sigma_{k}$, that is
\[
\sigma_{n+k}(z^{\prime},u^{\prime};z^{\prime\prime},u^{\prime\prime}%
)=\sigma_{n}(z^{\prime},z^{\prime\prime})+\sigma_{k}(u^{\prime},u^{\prime
\prime})
\]
for $(z^{\prime},u^{\prime}),(z^{\prime\prime},u^{\prime\prime})\in
\mathbb{R}^{2n}\times\mathbb{R}^{2k}$.\bigskip

\section{Symplectic Covariance of Weyl Calculus: Review}

For details and proofs we refer to Folland \cite{Folland}, de Gosson
\cite{Birk,Birkbis}, or Wong \cite{Wong}.

\subsection{Standard Weyl calculus}

In view of Schwartz kernel theorem all operators $A\in\mathcal{L}%
(\mathcal{S}(\mathbb{R}^{n}),\mathcal{S}^{\prime}(\mathbb{R}^{n}))$ admit the
representation
\begin{equation}
A\psi(x)=\langle K_{A}(x,\cdot),\psi(\cdot)\rangle\label{Ker}%
\end{equation}
where $K_{A}\in\mathcal{S}^{\prime}(\mathbb{R}^{n}\times\mathbb{R}^{n})$ and
$\langle\,,\,\rangle$ is the distributional bracket. The Weyl symbol of $A$ is
then
\begin{equation}
a(x,\xi_{x})=\int_{\mathbb{R}^{n}}e^{-i\xi_{x}\cdot y}K_{A}(x+\tfrac{1}%
{2}y,x-\tfrac{1}{2}y)\,dy \label{Ker-Sym}%
\end{equation}
with inverse
\begin{equation}
K_{A}(x,y)=\left(  \frac{1}{2\pi}\right)  ^{n}\int_{\mathbb{R}^{n}}%
e^{i\xi\cdot(x-y)}a(\frac{x+y}{2},\xi)\,d\xi\, . \label{Sym-Ker}%
\end{equation}
These integrals are well defined for $a\in\mathcal{S}(\mathbb{R}^{2n})
\Longleftrightarrow K_{A} \in\mathcal{S}(\mathbb{R}^{2n})$, and should
otherwise be interpreted as generalized Fourier (or inverse Fourier)
transforms (i.e. in the sense of distributions).

Substituting (\ref{Sym-Ker}) in (\ref{Ker}) and writing the distributional
bracket as an integral we obtain the standard formula for Weyl
pseudo-differential operators
\begin{equation}
{A}\psi(x)=\frac{1}{(2\pi)^{n}}\int_{\mathbb{R}^{2n}}e^{i(x-y)\cdot\xi}%
a(\frac{1}{2}(x+y),\xi)\psi(y)\,dy d\xi\,. \label{WeylPDO}%
\end{equation}

\subsection{The metaplectic group}

The symplectic group $\operatorname*{Sp}(2n,\mathbb{R})$ of $\mathbb{R}%
^{2n}\equiv\mathbb{R}^{n}\times\mathbb{R}^{n}$ consists of all linear
automorphisms $s$ of $\mathbb{R}^{2n}$ such that $\sigma_{n}(s(z),s(z^{\prime
}))=\sigma_{n}(z,z^{\prime})$ for all $z,z^{\prime}\in\mathbb{R}^{2n}$. The
group $\operatorname*{Sp}(2n,\mathbb{R})$ is a connected classical Lie group;
it has covering groups of all orders. Its double cover admits a faithful (but
not irreducible) representation by a group of unitary operators on
$L^{2}(\mathbb{R}^{n})$, the metaplectic group $\operatorname*{Mp}%
(2n,\mathbb{R})$. Thus, to every $s\in\operatorname*{Sp}(2n,\mathbb{R})$ one
can associate two unitary operators $S$ and $-S\in\operatorname*{Mp}%
(2n,\mathbb{R})$. One shows (Leray \cite{Leray}, de Gosson \cite{wiley,Birk})
that every element of $\operatorname*{Mp}(2n,\mathbb{R})$ is the product of
exactly two Fourier integral operators of the type
\begin{equation}
S_{W,.m}f(x)=\left(  \tfrac{1}{2\pi i}\right)  ^{n/2}\Delta_{m}(W)\int%
_{\mathbb{R}^{n}}e^{iW(x,x^{\prime})}f(x^{\prime})dx^{\prime} \label{swm}%
\end{equation}
for $f\in\mathcal{S}(\mathbb{R}^{n})$ where%
\[
W(x,x^{\prime})=\tfrac{1}{2}Px\cdot x-Lx\cdot x^{\prime}+\tfrac{1}%
{2}Qx^{\prime}\cdot x^{\prime}%
\]
is a quadratic form with $P=P^{T}$, $Q=Q^{T}$ and $\det L\neq0$, and
\begin{equation}
\Delta_{m}(W)=i^{m}\sqrt{|\det L|} \label{deltaw}%
\end{equation}
where $m$ is an integer (called the Maslov index) corresponding to a choice of
$\arg\det L$. The operator $S_{W,.m}$ belongs itself to $\operatorname*{Mp}%
(2n,\mathbb{R})$ and corresponds to the element $s\in\operatorname*{Sp}%
(2n,\mathbb{R})$ characterized by the condition%
\begin{equation}
(x,\xi)=s(x^{\prime},\xi^{\prime})\Longleftrightarrow\left\{
\begin{array}
[c]{c}%
\xi=\partial_{x}W(x,x^{\prime})\\
\xi^{\prime}=-\partial_{x^{\prime}}W(x,x^{\prime})
\end{array}
\right.  . \label{wgen}%
\end{equation}
It easily follows from the form of the generators (\ref{swm}) of
$\operatorname*{Mp}(2n,\mathbb{R})$ that metaplectic operators are continuous
mappings $\mathcal{S}(\mathbb{R}^{n})\longrightarrow\mathcal{S}(\mathbb{R}%
^{n})$ which extend by duality to continuous mappings $\mathcal{S}^{\prime
}(\mathbb{R}^{n})\longrightarrow\mathcal{S}^{\prime}(\mathbb{R}^{n})$. The
inverse of the operator $S_{W,m}$ is given by $S_{W,m}^{-1}=S_{W,m}^{\ast
}=S_{W^{\ast},m^{\ast}}$ where $W^{\ast}(x,x^{\prime})=-W(x^{\prime},x)$ and
$m^{\ast}=n-m$.

\subsection{Symplectic covariance}

The following important property is characteristic of Weyl pseudo-differential calculus:

\begin{proposition}
\label{prop1}Let $a\in\mathcal{S}^{\prime}(\mathbb{R}^{2n})$ and let
$A=\operatorname*{Op}^{w}(a)$ be the corresponding Weyl pseudo-differential
operator. Let $s\in\operatorname*{Sp}(2n,\mathbb{R)}$. We have%
\begin{equation}
\operatorname*{Op}\nolimits^{w}(a\circ s)=S^{-1}\operatorname*{Op}%
\nolimits^{w}(a)S \label{syco1}%
\end{equation}
where $S$ is anyone of the elements of $\operatorname*{Mp}(2n,\mathbb{R})$
corresponding to $s$.
\end{proposition}

This property has the following essential consequence for the
symplectic dimensional extensions. We will denote by
$\widetilde{S}$ the elements of
$\operatorname*{Mp}(2(n+k),\mathbb{R})$ because these operators
act on the "extended" space $L^{2}(\mathbb{R}^{n+k})$.

\begin{proposition}
\label{prop2}Let ${s}\in\operatorname*{Sp}(2(n+k),\mathbb{R)}$ and
$a\in\mathcal{S}^{\prime}(\mathbb{R}^{2n})$. Let $\widetilde{A}_{I}%
={\mathbb{E}}_{I}[A]$ and $\widetilde{A}_{{s}}={\mathbb{E}}_{s}[A]$ be the
corresponding symplectic dimensional extensions of $A=\operatorname*{Op}\nolimits^{w}%
(a)$. We have%
\begin{equation}
\widetilde{A}_{{s}}=\widetilde{S}^{-1}\widetilde{A}_{I}\widetilde{S}
\label{syco1tilde}%
\end{equation}
where $\widetilde{S}\in\operatorname*{Mp}(2(n+k),\mathbb{R)}$ is any of the
two metaplectic operators that projects onto ${s}$.
\end{proposition}

\begin{proof}
In view of the definition of the dimensional extensions $\widetilde{A}%
_{I}={\mathbb{E}}_{I}[A]$ and $\widetilde{A}_{{s}}={\mathbb{E}}_{ {s}}[A]$
formula (\ref{syco1tilde}) is equivalent to
\begin{equation}
\operatorname*{Op}\nolimits^{w}[(a\otimes1_{2k})\circ{s}]=\widetilde{S}%
^{-1}\operatorname*{Op}\nolimits^{w}((a\otimes1_{2k}))\widetilde{S}
\label{syco2}%
\end{equation}
and the latter is a straightforward consequence of Proposition \ref{prop1}.
\end{proof}

From this property it is easy to deduce some continuity properties for the
operators $\widetilde{A}_{{s}}$ knowing those of $\widetilde{A}_{I}$. Suppose
for instance that $\widetilde{A}_{I}$ is bounded on $L^{2}(\mathbb{R}^{n+k})$
then so is $\widetilde{A}_{{s}}$ since the metaplectic operators
$\widetilde{S}\in\operatorname*{Mp}(2(n+k),\mathbb{R)}$ are bounded on
$L^{2}(\mathbb{R}^{n+k})$. It actually suffices in this case to assume that
the operator $A$ is bounded on $L^{2}(\mathbb{R}^{n})$ as will follow from the
intertwining results of the next section.

\section{Properties of the Dimensional Extensions}

\subsection{A redefinition of $\protect\widetilde{A}$}

We begin by making a few general observations. Assume that $a\in
\mathcal{S}(\mathbb{R}^{2n})$, $A \overset{\text{Weyl}}{\longleftrightarrow}
a$ and let $\widetilde{A}={\mathbb{E}}_{I}[A]$. Then, for $\Psi\in
\mathcal{S}(\mathbb{R}^{n+k})$ we have from (\ref{WeylPDO})%
\begin{equation}
\widetilde{A}\Psi(x,y)=\left(  \tfrac{1}{2\pi}\right)  ^{n+k}\int%
_{\mathbb{R}^{2(n+k)}}e^{i[(x-x^{\prime})\xi+(y-y^{\prime})\eta]}a(\tfrac
{1}{2}(x+x^{\prime}),\xi)\Psi(x^{\prime},y^{\prime})dx^{\prime}dy^{\prime}d\xi
d\eta\label{atilda}%
\end{equation}
where the integral in $\eta$ is viewed as the inverse Fourier transform of
$1$; we thus have
\begin{equation}
\widetilde{A}\Psi(x,y)=\left(  \tfrac{1}{2\pi}\right)  ^{n}\int_{\mathbb{R}%
^{2n}}e^{i(x-x^{\prime})\xi}a(\tfrac{1}{2}(x+x^{\prime}),\xi)\Psi(x^{\prime
},y)dx^{\prime}d\xi\label{form}%
\end{equation}
which we can write in compact form as%
\begin{equation}
\widetilde{A}\Psi(\cdot,y)=A\psi_{y}\text{ \ with \ }\psi_{y}=\Psi(\cdot,y).
\label{aa1}%
\end{equation}
In particular, if $\psi\in\mathcal{S}(\mathbb{R}^{n})$ and $\phi\in
\mathcal{S}(\mathbb{R}^{k})$ then%
\begin{equation}
\widetilde{A}(\psi\otimes\phi)=(A\psi)\otimes\phi. \label{aa2}%
\end{equation}

These results can be extended to the more general case $a\in\mathcal{S}%
^{\prime}(\mathbb{R}^{2n})$. In view of formula (\ref{Ker}) we have
\begin{equation}
\widetilde{A}\Psi(z)=\langle K_{\widetilde{A}}(z,\cdot),\Psi(\cdot
)\rangle\label{Ext-Ker}%
\end{equation}
where $z=(x,y)$ and $K_{\widetilde{A}} \in\mathcal{S}^{\prime}(\mathbb{R}%
^{n+k}\times\mathbb{R}^{n+k})$ is given by
\begin{align}
K_{\widetilde{A}}(x,y;x^{\prime},y^{\prime})  &  =\left(  \tfrac{1}{2\pi
}\right)  ^{n+k}\int_{\mathbb{R}^{n+k}}e^{i\xi\cdot(x-x^{\prime}%
)+i\eta(y-y^{\prime})}a(\frac{x+x^{\prime}}{2},\xi)\,d\xi d\eta\nonumber\\
&  =K_{A}(x;x^{\prime})\delta(y-y^{\prime}) \label{Sym-ExtKer}%
\end{align}
where the Fourier transform is interpreted in the sense of distributions and
$\delta$ is the Dirac measure. It follows from (\ref{Sym-ExtKer}) that
\begin{align}
\widetilde{A}\Psi(x,y)  &  =\langle K_{A}(x,\cdot),\Psi(\cdot,y)\rangle
\nonumber\\
&  =\langle K_{A}(x,\cdot),\psi_{y}(\cdot)\rangle=A\psi_{y}(x) \label{aa3}%
\end{align}
where, once again, $\psi_{y}=\Psi(\cdot,y)$. In particular, if $\Psi
=\psi\otimes\phi$ with $\psi\in\mathcal{S}(\mathbb{R}^{n})$ and $\phi
\in\mathcal{S}(\mathbb{R}^{k})$, we have again
\begin{equation}
\widetilde{A}(\psi\otimes\phi)=(A\psi)\otimes\phi\,. \label{aa4}%
\end{equation}
This intertwining relation will be extended to the general case $\widetilde{A}%
=\mathbb{E}_{s}[A]$ in the next section.

Another interesting formulation of $\widetilde{A}=\mathbb{E}_{I}[A]$, also
valid in the case $a\in\mathcal{S}^{\prime}(\mathbb{R}^{2n})$, is as follows.
For $\Psi,\Phi\in\mathcal{S}(\mathbb{R}^{n+k})$ the cross-Wigner transform is
defined by \cite{Folland,Birk,Birkbis}%
\begin{align}
\label{wig1}W(\Psi,\Phi)(x,y;\xi,\eta)  &  =\left(  \tfrac{1}{2\pi}\right)
^{n+k}\int_{\mathbb{R}^{n}\times\mathbb{R}^{k}}e^{-i(\xi\cdot x^{\prime}%
+\eta\cdot y^{\prime})}\\
&  \times\Psi((x,y)+\tfrac{1}{2}(x^{\prime},y^{\prime}))\overline
{\Phi((x,y)-\tfrac{1}{2}(x^{\prime},y^{\prime}))}dx^{\prime}dy^{\prime
}\nonumber
\end{align}
and formula (\ref{atilda}) is equivalent to%
\begin{equation}
(\widetilde{A}\Psi|\Phi)_{L^{2}(\mathbb{R}^{n+k})}=\int_{\mathbb{R}^{2(n+k)}%
}a(x,\xi)W(\Psi,\Phi)(x,y;\xi,\eta)dxdyd\xi d\eta. \label{atildabis}%
\end{equation}
Note that since $W(\Psi,\Phi)\in\mathcal{S}(\mathbb{R}^{2(n+k)})$ the integral
above makes sense not only when $a\in\mathcal{S}(\mathbb{R}^{2n})$, but also
when $a$ is measurable and does not increase too fast at infinity. In fact,
viewing the integral as a distributional bracket, we can rewrite
(\ref{atildabis}) as%
\begin{equation}
\langle\widetilde{A}\Psi,\overline{\Phi}\rangle=\langle\widetilde{a}%
,W(\Psi,\Phi)\rangle\label{atildater}%
\end{equation}
which makes sense for every $\widetilde{a}=E_{I}[a]\in\mathcal{S}^{\prime
}(\mathbb{R}^{2(n+k)})$. More generally, we can redefine the operator
$\widetilde{A}_{{s}}=\widetilde{S}^{-1}\widetilde{A}\widetilde{S}%
=\mathbb{E}_{s}[A]$ by the formula%
\begin{equation}
\langle\widetilde{A}_{{s}}\Psi,\overline{\Phi}\rangle=\langle\widetilde{a}%
\circ s,W(\Psi,\Phi)\rangle\label{atilda4}%
\end{equation}
or, equivalently, by%
\begin{equation}
\langle\widetilde{A}_{{s}}\Psi,\overline{\Phi}\rangle=\langle\widetilde{a}%
,W(\widetilde{S}\Psi,\widetilde{S}\Phi)\rangle. \label{atilda5}%
\end{equation}

\subsection{Composition, invertibility and adjoints}

We will need the following result:

\begin{lemma}
\label{lem}If $A$ is a continuous operator $\mathcal{S}(\mathbb{R}%
^{n})\longrightarrow\mathcal{S}(\mathbb{R}^{n})$, then $\widetilde{A}%
=\mathbb{E}_{s}[A]$ is a continuous operator $\mathcal{S}(\mathbb{R}%
^{n+k})\longrightarrow\mathcal{S}(\mathbb{R}^{n+k}).$
\end{lemma}

\begin{proof}
We first remark that it is sufficient to prove this result for ${s}=I$. This
follows from the conjugation formula (\ref{syco1tilde}) in Proposition
\ref{prop2}: if $\widetilde{A}:\mathcal{S}(\mathbb{R}^{n+k})\longrightarrow
\mathcal{S}(\mathbb{R}^{n+k})$ then $\widetilde{S}^{-1}\widetilde{A}%
\widetilde{S}:\mathcal{S}(\mathbb{R}^{n+k})\longrightarrow\mathcal{S}%
(\mathbb{R}^{n+k})$ since metaplectic operators preserve the Schwartz classes;
the continuity statement follows likewise since metaplectic operators are
unitary. Assume that $A$ is a continuous operator $\mathcal{S}(\mathbb{R}%
^{n})\longrightarrow\mathcal{S}(\mathbb{R}^{n})$, and let $\widetilde{A}%
=\mathbb{E}_{I}[A]$ and $\Psi\in\mathcal{S}(\mathbb{R}^{n+k}) $. Then, in view
of formula (\ref{aa1}), $\widetilde{A}$ is also continuous and $\widetilde{A}%
\Psi\in\mathcal{S}(\mathbb{R}^{n+k})$.
\end{proof}

Recall that if $A\overset{\text{Weyl}}{\longleftrightarrow}a$ then
the formal adjoint of $A$ corresponds to the complex conjugate of
$a$, that is we have
$A^{\ast}\overset{\text{Weyl}}{\longleftrightarrow}\overline{a}$.
In particular $A$ is (formally) self-adjoint if and only if its
Weyl symbol is real. This property carries over to the case of
symplectic dimensionally extended operators:

\begin{proposition}
\label{propSA} We have:
\begin{equation}
{\mathbb{E}}_{s}[A]^{\ast}={\mathbb{E}}_{ {s}}[A^{\ast}]\,. \label{asas}%
\end{equation}
Hence, in particular, $\widetilde{A}=\mathbb{E}_{s}[A]$ is (formally)
self-adjoint if and only if $A$ is.
\end{proposition}

\begin{proof}
This property immediately follows from the definitions since we have
$\mathbb{E}_{s}[A]\overset{\text{Weyl}}{\longleftrightarrow}(a\otimes
1_{2k})\circ{s}$ and noting that the complex conjugate of $(a\otimes
1_{2k})\circ{s}$ is $(\overline{a}\otimes1_{2k})\circ{s}$.
\end{proof}

The proof of the following composition property is slightly more technical; it
shows that the extension mapping ${\mathbb{E}}_{s}$ preserves products of operators:

\begin{proposition}
Let $A\overset{\text{Weyl}}{\longleftrightarrow}a$ and $B\overset{\text{Weyl}%
}{\longleftrightarrow}b$ and assume that $B:\mathcal{S}(\mathbb{R}%
^{n})\longrightarrow\mathcal{S}(\mathbb{R}^{n})$. In this case $A$ and $B$ can
be composed. Then $\widetilde{A}={\mathbb{E}}_{s}[A]$ and $\widetilde{B}%
={\mathbb{E}}_{s}[B]$ can also be composed and we have:
\begin{equation}
{\mathbb{E}}_{s}[AB]={\mathbb{E}}_{{s}}[A]\, {\mathbb{E}}_{s}[B].
\label{absab}%
\end{equation}

\end{proposition}

\begin{proof}
That $\widetilde{A}\widetilde{B}$ is well-defined follows from Lemma
\ref{lem}. Recall that the Weyl symbol of the product $C=AB$ is given by the
formula%
\begin{equation}
c(z)=\left(  \tfrac{1}{4\pi}\right)  ^{2n}\int_{\mathbb{R}^{4n}}e^{\frac{i}%
{2}\sigma_{n}(z^{\prime},z^{\prime\prime})}a(z+\tfrac{1}{2}z^{\prime
})b(z-\tfrac{1}{2}z^{\prime\prime})dz^{\prime}dz^{\prime\prime} \label{cab1}%
\end{equation}
(see for instance \cite{Birkbis}, p.148). Let us first assume that ${s}$ is
the identity; then
\begin{align*}
\widetilde{a}(z+\tfrac{1}{2}z^{\prime},u+\tfrac{1}{2}u^{\prime})  &
=a(z+\tfrac{1}{2}z^{\prime})\\
\widetilde{b}(z-\tfrac{1}{2}z^{\prime},u-\tfrac{1}{2}u^{\prime})  &
=b(z-\tfrac{1}{2}z^{\prime})
\end{align*}
and the Weyl symbol $\widetilde{c}$ of $\widetilde{C}=\widetilde{A}%
\widetilde{B}$ is thus given by%
\begin{align*}
\widetilde{c}(z,u)  &  =\left(  \tfrac{1}{4\pi}\right)  ^{2(n+k)}%
\int_{\mathbb{R}^{4(n+k)}}e^{\frac{i}{2}\sigma_{n+k}(z^{\prime},u^{\prime
};z^{\prime\prime},u^{\prime\prime})}a(z+\tfrac{1}{2}z^{\prime})b(z-\tfrac
{1}{2}z^{\prime\prime})dz^{\prime}dz^{\prime\prime}du^{\prime}du^{\prime
\prime}\\
&  =\left(  \tfrac{1}{4\pi}\right)  ^{2(n+k)}\int_{\mathbb{R}^{4k}}e^{\frac
{i}{2}\sigma_{k}(u^{\prime},u^{\prime\prime})}du^{\prime}du^{\prime\prime}%
\int_{\mathbb{R}^{4n}}e^{\frac{i}{2}\sigma_{n}(z^{\prime},z^{\prime\prime}%
)}a(z+\tfrac{1}{2}z^{\prime})b(z-\tfrac{1}{2}z^{\prime\prime})dz^{\prime
}dz^{\prime\prime}\\
&  =\left(  \tfrac{1}{4\pi}\right)  ^{2n}\int_{\mathbb{R}^{4n}}e^{\frac{i}%
{2}\sigma_{n}(z^{\prime},z^{\prime\prime})}a(z+\tfrac{1}{2}z^{\prime
})b(z-\tfrac{1}{2}z^{\prime\prime})dz^{\prime}dz^{\prime\prime}%
\end{align*}
where we have used the relation $\sigma_{n+k}=\sigma_{n}\oplus\sigma_{k}$ and
the Fourier inversion formula to compute the integral in $u^{\prime}%
,u^{\prime\prime}$; thus $\widetilde{c}(z,u)=c(z)$ which proves our claim in
the case ${s}=I$. The general case follows from formula (\ref{syco1tilde}) of
Proposition \ref{prop2}.
\end{proof}

The next result is a corollary of the previous proposition.

\begin{corollary}
\label{CorIN} Let $A^{-1}:\mathcal{S}(\mathbb{R}^{n}) \longrightarrow
\mathcal{S}(\mathbb{R}^{n})$ be the inverse of $A:\mathcal{S}(\mathbb{R}^{n})
\longrightarrow\mathcal{S}(\mathbb{R}^{n})$, i.e. $AA^{-1}=A^{-1}A=I_{n}$
where $I_{n}$ is the identity operator in $\mathcal{S}(\mathbb{R}^{n})$. Then
$\mathbb{E}_{s}[A^{-1}]$ is the inverse of $\mathbb{E}_{s}[A]$, that is
\begin{equation}
\label{Inverse}\mathbb{E}_{s}[A^{-1}]=\left(  \mathbb{E}_{s}[A]\right)  ^{-1}.
\end{equation}

\end{corollary}

\begin{proof}
We first notice that $\mathbb{E}_{s}[I_{n}]=I_{n+k}$, which follows directly
from (\ref{syco1tilde}) and (\ref{aa1}). Then from (\ref{absab})
\begin{equation}
\mathbb{E}_{s}[A]\mathbb{E}_{s}[A^{-1}]=\mathbb{E}_{s}[AA^{-1}]=\mathbb{E}%
_{s}[I_{n}]=I_{n+k}%
\end{equation}
and likewise $\mathbb{E}_{s}[A^{-1}]\mathbb{E}_{s}[A]=I_{n+k}$. Hence
\[
\mathbb{E}_{s}[A^{-1}]=\left(  \mathbb{E}_{s}[A]\right)  ^{-1}.
\]

\end{proof}

We notice for future reference that the composition and invertibility results
extend by duality to $\mathcal{S}^{\prime}(\mathbb{R}^{n+k})$.

\begin{remark}
\label{remark-duality} If $A:\mathcal{S}(\mathbb{R}^{n}) \longrightarrow
\mathcal{S}(\mathbb{R}^{n})$ is linear and continuous then it extends by
duality to a continuous map that we shall denote by
\begin{equation}
\label{A-extension}A^{\prime}:\mathcal{S}^{\prime}(\mathbb{R}^{n})
\longrightarrow\mathcal{S}^{\prime}(\mathbb{R}^{n}).
\end{equation}
We have, of course, $\left.  A^{\prime}\right|  _{\mathcal{S}(\mathbb{R}^{n}%
)}=A$.

In this case, $\widetilde{A}=\mathbb{E}_{s}[A]$ is also a linear and
continuous map $\mathcal{S}(\mathbb{R}^{n+k}) \longrightarrow\mathcal{S}%
(\mathbb{R}^{n+k})$ (Lemma \ref{lem}) and so it also extends to a continuous
map
\[
\widetilde{A}^{\prime}:\mathcal{S}^{\prime}(\mathbb{R}^{n+k}) \longrightarrow
\mathcal{S}^{\prime}(\mathbb{R}^{n+k})
\]
Since both $A^{\prime}$ and $\widetilde{A}^{\prime}$ are uniquely defined, we
may extend the action of $\mathbb{E}$ to the operators $A^{\prime}%
:\mathcal{S}^{\prime}(\mathbb{R}^{n}) \longrightarrow\mathcal{S}^{\prime
}(\mathbb{R}^{n})$ by setting
\begin{equation}
\label{Ext-E}\mathbb{E}_{s}[A^{\prime}]:= \widetilde{A}^{\prime}.
\end{equation}
Let $B$ be also a continuous map $\mathcal{S}(\mathbb{R}^{n})\longrightarrow
\mathcal{S}(\mathbb{R}^{n})$. It then follows trivially from the definition
(\ref{Ext-E}) that
\begin{equation}
\mathbb{E}_{s}[A^{\prime}B^{\prime}]=\mathbb{E}_{s}[A^{\prime}]\mathbb{E}%
_{s}[B^{\prime}]
\end{equation}
since $A^{\prime}B^{\prime}=(AB)^{\prime}$; and that
\begin{equation}
\mathbb{E}_{s}[A^{\prime-1}]=\left(  \mathbb{E}_{s}[A^{\prime}]\right)  ^{-1}%
\end{equation}
since $A^{\prime-1}=\left(  A^{-1}\right)  ^{\prime}$.

We conclude that the three maps $^{\prime}\, , \, \mathbb{E}_{s}$ and $^{-1}$
commute with each other in the space of linear, continuous and invertible
operators $\mathcal{S}(\mathbb{R}^{n}) \longrightarrow\mathcal{S}%
(\mathbb{R}^{n})$.
\end{remark}

\subsection{Intertwiners}

For $\chi\in\mathcal{S}(\mathbb{R}^{k})$ and $\widetilde{S}\in
\operatorname*{Mp}(2(n+k),\mathbb{R)}$ let us define the operator
$T_{\widetilde{S},\chi}:\mathcal{S}(\mathbb{R}^{n})\longrightarrow
\mathcal{S}(\mathbb{R}^{n+k})$ by the formula%
\begin{equation}
T_{\widetilde{S},\chi}\psi=\widetilde{S}^{-1}(\psi\otimes\chi). \label{tics}%
\end{equation}
These operators $T_{\widetilde{S},\chi}$ have the following analytical properties:

\begin{proposition}
\label{proptic2}The operator $T_{\widetilde{S},\chi}$ extends to a continuous
map $\mathcal{S}^{\prime}(\mathbb{R}^{n})\longrightarrow\mathcal{S}^{\prime
}(\mathbb{R}^{n+k})$ and, if $||\chi||=1$, to a partial isometry
$L^{2}(\mathbb{R}^{n})\longrightarrow L^{2}(\mathbb{R}^{n+k})$.

The adjoint $T_{\widetilde{S},\chi}^{\ast}:L^{2}(\mathbb{R}^{n+k}%
)\longrightarrow L^{2}(\mathbb{R}^{n})$ is given by the formula%
\begin{equation}
T_{\widetilde{S},\chi}^{\ast}\Phi(x)=\int_{\mathbb{R}^{k}}\overline{\chi(y)}
{\widetilde{S}\Phi(x,y)}\, dy \label{tadj}%
\end{equation}
and also extends to the continuous operator $\mathcal{S}^{\prime}%
(\mathbb{R}^{n+k})\longrightarrow\mathcal{S}^{\prime}(\mathbb{R}^{n})$ defined
by
\begin{equation}
\langle T_{\widetilde{S},\chi}^{\ast} \Phi, \psi\rangle=\langle{\widetilde{S}
\Phi} , \overline{\chi} \otimes\psi\rangle\quad, \quad\forall\psi
\in\mathcal{S}(\mathbb{R}^{n}). \label{tadj2}%
\end{equation}

\end{proposition}

\begin{proof}
Let $\psi\in\mathcal{S}^{\prime}(\mathbb{R}^{n})$; then $\psi\otimes\chi
\in\mathcal{S}^{\prime}(\mathbb{R}^{n+k})$ and hence $\widetilde{S}^{-1}%
(\psi\otimes\chi)\in\mathcal{S}^{\prime}(\mathbb{R}^{n+k})$. The formula
$T_{\widetilde{S},\chi}\psi=\widetilde{S}^{-1}(\psi\otimes\chi)$ defines the
desired extension. Assume now $||\chi||=1$ and $\psi\in L^{2}(\mathbb{R}^{n}%
)$; dropping for notational simplicity the references to the spaces in scalar
products and norms we have%
\[
||T_{\widetilde{S},\chi}\psi||^{2}=||\widetilde{S}^{-1}(\psi\otimes\chi
)||^{2}=||\psi\otimes\chi||^{2}%
\]
since $\widetilde{S}$ is unitary. Now $||\psi\otimes\chi||^{2}=||\psi
||^{2}||\chi||^{2}$ hence $T_{\widetilde{S},\chi}$ is a partial isometry. Its
adjoint is defined by $(T_{\widetilde{S},\chi}^{\ast}\Phi|\psi)=(\Phi
|T_{\widetilde{S},\chi}\psi)$. Formula (\ref{tadj}) follows since we have%
\begin{align*}
(\Phi|T_{\widetilde{S},\chi}\psi)  &  =(\Phi|\widetilde{S}^{-1}(\psi
\otimes\chi))\\
&  =(\widetilde{S}\Phi|\psi\otimes\chi)\\
&  =\int_{\mathbb{R}^{n}}\overline{\psi(x)}\left[  \int_{\mathbb{R}^{k}%
}\overline{\chi(y)}{\widetilde{S}\Phi(x,y)}dy\right]  dx.
\end{align*}
Finally, using distributional brackets and the definition of the adjoint, the
previous equation can be re-written as
\[
\langle T_{\widetilde{S},\chi}^{\ast} \Phi, \psi\rangle=\langle\widetilde{S}
\Phi, \overline{\chi} \otimes\psi\rangle\quad, \quad\forall\psi\in
\mathcal{S}(\mathbb{R}^{n})
\]
which defines the distribution $T_{\widetilde{S},\chi}^{\ast} \Phi
\in\mathcal{S}^{\prime}(\mathbb{R}^{n})$ for every $\Phi\in\mathcal{S}%
^{\prime}(\mathbb{R}^{n+k})$.
\end{proof}

The following intertwining relations are essential for studying
the spectral properties of the symplectic dimensional extensions:

\begin{proposition}
\label{proptic1}Let $s\in\operatorname*{Sp}(2(n+k),\mathbb{R})$ and
$\widetilde{S}$ be one of the two metaplectic operators that projects onto
$s$. Then
\begin{equation}
\widetilde{A}_{{s}}T_{\widetilde{S},\chi}=T_{\widetilde{S},\chi}A\text{ \ ,
\ }T_{\widetilde{S},\chi}^{\ast}\widetilde{A}_{ {s}}=AT_{\widetilde{S},\chi
}^{\ast} \label{inter}%
\end{equation}
for every $A=\operatorname*{Op}\nolimits^{w}(a)$, $a\in\mathcal{S}^{\prime
}(\mathbb{R}^{2n})$ and $\widetilde{A}_{s}={\mathbb{E}}_{s}[A]$.
\end{proposition}

\begin{proof}
Let us first prove that $\widetilde{A}_{{s}}T_{\widetilde{S},\chi
}=T_{\widetilde{S},\chi}A$ for the case $\widetilde{S}=I$. Let us set
$T_{\chi}=T_{I,\chi}$ and let $\psi\in\mathcal{S}(\mathbb{R}^{n})$. In view of
formula (\ref{aa3}) we have%
\begin{align*}
\widetilde{A}_{I}(T_{\chi}\psi)(x,y)  &  =A(T_{\chi}\psi)_{y}(x)\\
&  =A\psi(x)\chi(y)\\
&  =T_{\chi}(A\psi)(x,y).
\end{align*}
The general case follows using formula (\ref{syco1tilde}) since%
\begin{align*}
\widetilde{A}_{{s}}T_{\widetilde{S},\chi}  &  =\widetilde{S}^{-1}%
\widetilde{A}_{I}\widetilde{S}(\widetilde{S}^{-1}T_{\chi})=\widetilde{S}%
^{-1}\widetilde{A}_{I}T_{\chi}\\
&  =\widetilde{S}^{-1}T_{\chi}A=T_{\widetilde{S},\chi}A.
\end{align*}
The second formula (\ref{inter}) is deduced from the first: since
$\widetilde{A}_{{s}}^{\ast}T_{\widetilde{S},\chi}=T_{\widetilde{S},\chi
}A^{\ast}$ (because of the equality (\ref{asas})), we have $(\widetilde{A}_{
{s}}^{\ast}T_{\widetilde{S},\chi})^{\ast}=(T_{\widetilde{S},\chi}A^{\ast
})^{\ast}$ that is $T_{\widetilde{S},\chi}^{\ast}\widetilde{A}_{ {s}%
}=AT_{\widetilde{S},\chi}^{\ast}$.
\end{proof}

\section{Spectral Results}

\subsection{Orthonormal bases and the partial isometries
$T_{\protect\widetilde{S},\chi}$}

In view of Proposition \ref{proptic2} the intertwining operators
$T_{\widetilde{S},\chi}$ are partial isometries of $L^{2}(\mathbb{R}^{n})$
into $L^{2}(\mathbb{R}^{n+k})$. Hence the range $\mathcal{H}_{\widetilde{S}%
,\chi}$ of $T_{\widetilde{S},\chi}$ is a closed subspace of $L^{2}%
(\mathbb{R}^{n+k})$. The following result shows how the operators
$T_{\widetilde{S},\chi}$ generate orthonormal bases of $L^{2}(\mathbb{R}%
^{n+k})$.

\begin{proposition}
\label{propONB}Let $(\phi_{j})_{j}$ be an orthonormal basis of $L^{2}%
(\mathbb{R}^{n})$ and $(\chi_{l})_{l}$ an orthonormal basis of $L^{2}%
(\mathbb{R}^{k})$. Then $(T_{\widetilde{S},\chi_{l}}\phi_{j})_{j,l}$ is an
orthonormal basis of $L^{2}(\mathbb{R}^{n+k})$ and $L^{2}(\mathbb{R}%
^{n+k})=\oplus_{l} \mathcal{H}_{\widetilde{S},\chi_{l}}$.
\end{proposition}

\begin{proof}
Let us set $\Phi_{j,l}=T_{\widetilde{S},\chi_{l}}\phi_{j}$. We have%
\begin{align*}
(\Phi_{j,l}|\Phi_{j^{\prime},l^{\prime}})  &  =(\widetilde{S}^{-1}(\phi
_{j}\otimes\chi_{l})|\widetilde{S}^{-1}(\phi_{j^{\prime}}\otimes
\chi_{l^{\prime}}))\\
&  =(\phi_{j}\otimes\chi_{l}|\phi_{j^{\prime}}\otimes\chi_{l^{\prime}})\\
&  =(\phi_{j}|\phi_{j^{\prime}})(\chi_{l}|\chi_{l^{\prime}})
\end{align*}
hence the vectors $\Phi_{j,l}$ form an orthonormal system in $L^{2}%
(\mathbb{R}^{n+k})$. Let us prove that this system is complete in
$L^{2}(\mathbb{R}^{n+k})$; it suffices for that to show that if $\Psi\in
L^{2}(\mathbb{R}^{n+k})$ is such that $(\Psi|\Phi_{j,l})=0$ for all indices
$j,l$ then $\Psi=0$. Now,%
\[
(\Psi|\Phi_{j,l})=(\Psi|\widetilde{S}^{-1}(\phi_{j}\otimes\chi_{l}%
))=(\widetilde{S}\Psi|\phi_{j}\otimes\chi_{l}).
\]
and $(\Psi|\Phi_{j,l})=0$ for all $j,l$ implies $\widetilde{S}\Psi=0$ because
the tensor product of the two orthonormal bases is an orthonormal basis of
$L^{2}(\mathbb{R}^{n+k})$; it follows that $\Psi=0$ as claimed. Moreover, for
each $l$ the set $(\Phi_{j,l})_{j}$ is an orthonormal basis of $\mathcal{H}%
_{\widetilde{S},\chi_{l}}$. Since for $l \not =l^{\prime}$ we have
$\mathcal{H}_{\widetilde{S},\chi_{l}} \cap\mathcal{H}_{\widetilde{S}%
,\chi_{l^{\prime}}}=\{0\}$ it follows that $L^{2}(\mathbb{R}^{n+k})=\oplus_{l}
\mathcal{H}_{\widetilde{S},\chi_{l}}$.
\end{proof}

The observant Reader will have noticed that in the first part of the proof we
established the equality%
\begin{equation}
(T_{\widetilde{S},\chi}\phi|T_{\widetilde{S},\chi^{\prime}}\phi^{\prime
})_{L^{2}(\mathbb{R}^{n+k})}=(\phi|\phi^{\prime})_{L^{2}(\mathbb{R}^{n})}%
(\chi|\chi^{\prime})_{L^{2}(\mathbb{R}^{k})} \label{moylike}%
\end{equation}
which is valid for all pairs of functions $(\phi,\phi^{\prime})$
in\ $L^{2}(\mathbb{R}^{n})$ and $(\chi,\chi^{\prime})$ in\ $L^{2}%
(\mathbb{R}^{k})$.

\subsection{Spectral result}

Let us prove the main result of this section:

\begin{proposition}
\label{spectraltheor} Let $s\in$ $\operatorname*{Sp}(2(n+k),\mathbb{R})$, let
$A:\mathcal{S}(\mathbb{R}^{n})\mapsto\mathcal{S}^{\prime}(\mathbb{R}^{n})$ be
a linear continuous operator and let $\widetilde{A}_{s}={\mathbb{E}}_{s}[A]$.
Let $\widetilde{S}$ be one of the metaplectic operators associated with $s$. Then:

(i) The eigenvalues of $A$ and $\widetilde{A}_{{s}}$ are the same.

(ii) If $\psi$ is an eigenfunction of $A$ corresponding to the eigenvalue
$\lambda$ then $\Psi=T_{\widetilde{S},\chi}\psi$ is an eigenfunction of
$\widetilde{A}_{s}$ corresponding to $\lambda$, for every $\chi\in
\mathcal{S}(\mathbb{R}^{k}) \backslash\left\{  0 \right\}  $, and we have
$\Psi\in\mathcal{S}(\mathbb{R}^{n+k})$.

(iii) Conversely, if $\Psi$ is an eigenfunction of $\widetilde{A}_{{s}}$ and
$\psi=T_{\widetilde{S},\chi}^{*}\Psi$ is different from zero then $\psi$ is an
eigenfunction of $A$ and corresponds to the same eigenvalue.

(iv) If $(\psi_{j})_{j}$ is an orthonormal basis of eigenfunctions of $A$ and
$(\chi_{l}\in\mathcal{S}(\mathbb{R}^{k}))_{l}$ is an orthonormal basis of
$L^{2}(\mathbb{R}^{k})$ then $(T_{\widetilde{S},\chi_{l}}\psi_{j})_{j,l}$ is a
complete set of eigenfunctions of $\widetilde{A}_{s}$ and forms an orthonormal
basis of $L^{2}(\mathbb{R}^{n+k})$.
\end{proposition}

\begin{proof}
That every eigenvalue of $A$ also is an eigenvalue of $\widetilde{A}_{ {s}}$
is clear: if $A\psi=\lambda\psi$ for some $\psi\neq0$ then
\[
\widetilde{A}_{{s}}(T_{\widetilde{S},\chi}\psi)=T_{\widetilde{S},\chi}%
A\psi=\lambda(T_{\widetilde{S},\chi}\psi)
\]
and $T_{\widetilde{S},\chi}\psi\neq0$ because $T_{\widetilde{S},\chi}$ is
injective; this shows at the same time that $\Psi=T_{\widetilde{S},\chi}\psi$
is an eigenfunction of $\widetilde{A}_{{s}}$. Since $T_{\widetilde{S},\chi
}:\mathcal{S}(\mathbb{R}^{n})\longrightarrow\mathcal{S}(\mathbb{R}^{n+k})$ it
is clear that $\Psi\in\mathcal{S}(\mathbb{R}^{n+k})$ which concludes the proof
of (ii).

Assume conversely that $\widetilde{A}_{{s}}\Psi=\lambda\Psi$ and $\Psi\neq0$ .
For every $\chi$ we have, using the second intertwining relation
(\ref{inter}),
\[
AT_{\widetilde{S},\chi}^{\ast}\Psi=T_{\widetilde{S},\chi}^{\ast}%
\widetilde{A}_{{s}}\Psi=\lambda T_{\widetilde{S},\chi}^{\ast}\Psi
\]
hence $\lambda$ is an eigenvalue of $A$ and $T_{\widetilde{S},\chi}^{\ast}%
\Psi$ will be an eigenfunction of $A$ if it is different from zero. This
proves (iii).

To conclude the proof of (i) we still have to show that if $\Psi$ is a
eigenfunction of $\widetilde{A}_{s}$ then $T_{\widetilde{S},\chi}^{\ast}
\Psi\not = 0$ for some $\chi\in\mathcal{S}(\mathbb{R}^{k})$. We note that
$T_{\widetilde{S},\chi}T_{\widetilde{S},\chi}^{\ast}=P_{\widetilde{S},\chi}$
is the orthogonal projection onto the range $\mathcal{H}_{\widetilde{S},\chi}$
of $T_{\widetilde{S},\chi}$. Assume that $T_{\widetilde{S},\chi}^{\ast}\Psi=0$
for all $\chi\in\mathcal{S}(\mathbb{R}^{k})$; then $P_{\widetilde{S},\chi}%
\Psi=0$ for every $\chi\in\mathcal{S}(\mathbb{R}^{k})$, and hence $\Psi=0$ in
view of Proposition \ref{propONB}; but this is not possible since $\Psi$ is an eigenfunction.

Finally, the statement (iv) follows from (ii) and Proposition \ref{propONB}.
\end{proof}

\subsection{Generalized spectral theorem\label{subspec}}

If $A$ is a continuous linear operator $\mathcal{S}(\mathbb{R}^{n}%
)\longrightarrow\mathcal{S}(\mathbb{R}^{n})$ then it can be extended to a
continuous operator $\mathcal{S}^{\prime}(\mathbb{R}^{n})\longrightarrow
\mathcal{S}^{\prime}(\mathbb{R}^{n})$. In view of Lemma \ref{lem},
$\widetilde{A}_{s}={\mathbb{E}}_{s}[A]$ is also a continuous operator
$\mathcal{S}(\mathbb{R}^{n+k})\longrightarrow\mathcal{S}(\mathbb{R}^{n+k})$
and thus can also be extended to a continuous operator $\mathcal{S}^{\prime
}(\mathbb{R}^{n+k})\longrightarrow\mathcal{S}^{\prime}(\mathbb{R}^{n+k})$. In
this case Proposition \ref{spectraltheor} can be extended to a generalized
spectral result. Let $(\psi,\lambda)\in\mathcal{S}^{\prime}(\mathbb{R}%
^{n})\times\mathbb{C}$. We will say that $\psi$ is a generalized eigenvector
of $A$, corresponding to the generalized eigenvalue $\lambda$ if $\psi\not =0$
and
\[
\langle\psi,\overline{A^{\ast}\phi}\rangle=\lambda\langle\psi,\overline{\phi
}\rangle\text{ \ \textit{for all} }\phi\in\mathcal{S}(\mathbb{R}^{n}).
\]
We will write this equality formally as $(\psi|A^{\ast}\phi)=\lambda(\psi
|\phi) $ since both sides coincide with the usual scalar products when $\psi$
is a square integrable function.

\begin{proposition}
\label{GST} Assume that $A$ is a continuous linear operator of the form
$A:\mathcal{S}(\mathbb{R}^{n})\longrightarrow\mathcal{S}(\mathbb{R}^{n})$ and
let $\widetilde{A}_{s}={\mathbb{E}}_{s}[A]$. Then:

(i) The generalized eigenvalues of the operators $A$ and $\widetilde{A}_{{s}}$
are the same.

(ii) Let $\psi$ be a generalized eigenvector of $A$. Then, for every $\chi
\in\mathcal{S}(\mathbb{R}^{k}) \backslash\left\{  0 \right\}  $ the vector
$\Psi=T_{\widetilde{S},\chi} \psi$ is a generalized eigenvector of
$\widetilde{A}_{s}$ and corresponds to the same generalized eigenvalue.

(iii) Conversely, if $\Psi$ is a generalized eigenvector of $\widetilde{A}%
_{s}$ and $\psi=T_{\widetilde{S},\chi}^{\ast} \Psi\not =0$ then $\psi$ is a
generalized eigenvector of $A$ corresponding to the same generalized eigenvalue.
\end{proposition}

\begin{proof}
It is similar, \textit{mutatis mutandis}, to the proof of Theorem 445, Ch.19,
in \cite{Birkbis} (also see \cite{CPDE}). For instance, to prove (ii) one
notes that if $(\psi|A^{\ast}\phi)=\lambda(\psi|\phi)$ for every $\phi
\in\mathcal{S}(\mathbb{R}^{n})$ then also $(T_{\widetilde{S},\chi}%
\psi|\widetilde{A}^{\ast}_{s}\Phi)=\lambda(T_{\widetilde{S},\chi}\psi|\Phi)$
for every $\Phi\in\mathcal{S}(\mathbb{R}^{n+k})$. In fact, using the
intertwining property (\ref{inter}),%
\begin{align*}
(T_{\widetilde{S},\chi}\psi|\widetilde{A}^{\ast}_{s}\Phi)  &  =(\psi
|T_{\widetilde{S},\chi}^{\ast}\widetilde{A}^{\ast}_{s}\Phi)\\
&  =(\psi|(\widetilde{A}_{s}T_{\widetilde{S},\chi})^{\ast}\Phi)\\
&  =(\psi|(T_{\widetilde{S},\chi}A)^{\ast}\Phi)\\
&  =(\psi|A^{\ast}T_{\widetilde{S},\chi}^{\ast}\Phi)\\
&  =\lambda(\psi|T_{\widetilde{S},\chi}^{\ast}\Phi)
\end{align*}
hence the claim.
\end{proof}

\section{A new class of pseudo-differential operators\label{sectionSEC}}

In this section we define a new class of pseudo-differential operators which
is an extension of the well-known Shubin class $HG_{\rho}^{m_{1},m_{0}%
}(\mathbb{R}^{2n})$ of globally hypoelliptic operators. Moreover, we prove
several general results about the spectral, invertibility and hypoellipticity
properties of the operators in the new class.

\subsection{The Shubin classes of global hypoelliptic
operators\label{subsectionSC}}

A linear operator $A:\mathcal{S}^{\prime}(\mathbb{R}^{n})\longrightarrow
\mathcal{S}^{\prime}(\mathbb{R}^{n})$ is \textquotedblleft globally
hypoelliptic\textquotedblright\ if it satisfies%
\begin{equation}
\psi\in\mathcal{S}^{\prime}(\mathbb{R}^{n})\text{ and }A\psi\in\mathcal{S}%
(\mathbb{R}^{n})\Longrightarrow\psi\in\mathcal{S}(\mathbb{R}^{n}).
\label{glob1}%
\end{equation}
This notion is different with respect to the ordinary $C^{\infty}$
hypoellipticity (which is a \textit{local} property) because it
incorporates the decay at infinity of the involved functions or
distributions.

The Shubin class of globally hypoelliptic operators $HG_{\rho}^{m_{1},m_{0}%
}(\mathbb{R}^{2n})$ is a subclass of the Shubin class $G_{\rho}^{m_{1}%
}(\mathbb{R}^{2n})$. It is defined (Shubin \cite{sh87}, de Gosson
\cite{Birkbis}) as follows:

\begin{definition}
\label{defintionSC} Let $m_{0}$, $m_{1}$, and $\rho$ be real numbers such that
$m_{0}\leq m_{1}$ and $0<\rho\leq1$. The symbol class $H\Gamma_{\rho}%
^{m_{1},m_{0}}(\mathbb{R}^{2n})$ consists of all functions $a\in C^{\infty
}(\mathbb{R}^{2n})$ such that for $|z|$ sufficiently large the following
estimates hold:
\begin{equation}
C_{0}|z|^{m_{0}}\leq|a(z)|\leq C_{1}|z|^{m_{1}}%
\end{equation}
for some $C_{0},C_{1}>0$ and, for every $\alpha\in\mathbb{N}^{2n}$ there
exists $C_{\alpha}\geq0$ such that
\begin{equation}
|\partial_{z}^{\alpha}a(z)|\leq C_{\alpha}|a(z)||z|^{-\rho|\alpha|}.
\end{equation}
The Shubin class $HG_{\rho}^{m_{1},m_{0}}(\mathbb{R}^{2n})$ consists of all
operators $A:\mathcal{S}(\mathbb{R}^{n})\longrightarrow\mathcal{S}^{\prime
}(\mathbb{R}^{n})$ with Weyl symbols $a\in H\Gamma_{\rho}^{m_{1},m_{0}%
}(\mathbb{R}^{2n})$.
\end{definition}

The main appeal of the classes $HG_{\rho}^{m_{1},m_{0}}(\mathbb{R}^{2n})$ is
the set of general properties which is satisfied by their operators. Moreover,
they contain many operators which are important in mathematical physics (a
trivial example is the harmonic oscillator Hamiltonian). Let us then state the
main properties of the operators in $HG_{\rho}^{m_{1},m_{0}}(\mathbb{R}
^{2n})$ (for proofs see [\cite{sh87}, Chapter IV]).

Since $H\Gamma_{\rho}^{m_{1},m_{0}}(\mathbb{R}^{2n})\subset\Gamma_{\rho
}^{m_{1}}(\mathbb{R}^{2n})$, every operator $A\in HG_{\rho}^{m_{1},m_{0}%
}(\mathbb{R}^{2n})$ is a continuous operator $\mathcal{S}(\mathbb{R}%
^{n})\longrightarrow\mathcal{S}(\mathbb{R}^{n})$ which extends by duality to a
continuous operator $\mathcal{S}^{\prime}(\mathbb{R}^{n})\longrightarrow
\mathcal{S}^{\prime}(\mathbb{R}^{n})$. Moreover,

\begin{proposition}
[Shubin]\label{propshubin}Let $A\in HG_{\rho}^{m_{1},m_{0}}(\mathbb{R}^{2n})$
with $m_{0}>0$. If $A$ is formally self-adjoint, that is if $(\widehat{A}%
\psi|\phi)=(\psi|\widehat{A}\phi)$ for all test functions $\psi,\phi
\in\mathcal{S}(\mathbb{R}^{n})$, then $A$ is essentially self-adjoint and has
discrete spectrum in $L^{2}(\mathbb{R}^{n})$. Moreover there exists an
orthonormal basis of eigenfunctions $\phi_{j}\in\mathcal{S}(\mathbb{R}^{n})$
($j=1,2,...$) with eigenvalues $\lambda_{j}\in\mathbb{R}$ such that
$\lim_{j\rightarrow\infty}|\lambda_{j}|=\infty$.
\end{proposition}

In addition, every $A\in HG_{\rho}^{m_{1},m_{0}}(\mathbb{R}^{2n})$ is globally
hypoelliptic and, if $\operatorname*{Ker}A=\operatorname*{Ker}$ $A^{\ast
}=\{0\}$, it is also invertible with inverse $A^{-1} \in HG_{\rho}%
^{-m_{0},-m_{1}}(\mathbb{R}^{2n})$.

\subsection{The extended Shubin classes $\protect\widetilde{HG}_{\rho}%
^{m_{1},m_{0}}(\mathbb{R}^{2n})$\label{subsectionESC}}

We now define a new class of pseudo-differential operators:

\begin{definition}
\label{definitionESC}

The symbol class
$\widetilde{H\Gamma}_{\rho}^{m_{1},m_{0}}(\mathbb{R}^{2n})$
consists of all functions $\widetilde{a} \in
C^{\infty}(\mathbb{R}^{2n})$ for which there is $a
\in{H\Gamma}_{\rho}^{m_{1},m_{0}}(\mathbb{R}^{2(n-k)})$ (for some
$k \in\{0,1,..,n-1\}$) and a symplectic extension map
$E_{s}:\mathcal{S}^{\prime }(\mathbb{R}^{2(n-k)})
\longrightarrow\mathcal{S}^{\prime}(\mathbb{R}^{2n})$ such that
$\widetilde{a}=E_{s}[a]$. The set of pseudo-differential operators
with Weyl symbols $\widetilde{a} \in\widetilde{H\Gamma}_{\rho}^{m_{1},m_{0}%
}(\mathbb{R}^{2n})$ is called an "extended Shubin class" and denoted by
$\widetilde{HG}_{\rho}^{m_{1},m_{0}}(\mathbb{R}^{2n})$.
\end{definition}

We notice that, by construction, the symbols in the classes
$\widetilde{H\Gamma}_{\rho}^{m_{1},m_{0}}(\mathbb{R}^{2n})$ satisfy the
property
\[
\widetilde{a}\in\widetilde{H\Gamma}_{\rho}^{m_{1},m_{0}}(\mathbb{R}%
^{2n})\Longleftrightarrow\widetilde{a}\circ s\in\widetilde{H\Gamma}_{\rho
}^{m_{1},m_{0}}(\mathbb{R}^{2n}),\quad\forall s\in\operatorname*{Sp}%
(2n,\mathbb{R}).
\]
Apart from the trivial case $k=0$ or the case where $a(z)$ is a
polynomial, the symbol $E_s[a]$ does not belong any more to the
Shubin classes. Hence the Shubin results concerning
hypoellipticity and spectral theory do not apply to the operators
in the extended Shubin classes of Definition \ref{definitionESC}.

In the next section we will present several examples of operators in the
classes $\widetilde{HG}^{m_{1},m_{0}}_{\rho}(\mathbb{R}^{2n})$, which are
important for applications in quantum mechanics and deformation quantization.
First, let us study the general properties of the operators in $\widetilde{HG}%
^{m_{1},m_{0}}_{\rho}(\mathbb{R}^{2n})$.

\begin{lemma}
\label{Kernelrelation} Let $\widetilde{A} \in\widetilde{HG}_{\rho}%
^{m_{1},m_{0}}(\mathbb{R}^{2n})$ and let $A \in{HG}_{\rho}^{m_{1},m_{0}%
}(\mathbb{R}^{2(n-k)})$ be the associated operator such that $\widetilde{A}
=\mathbb{E}_{s}[A]$. Then Ker $\widetilde{A}=\{0\}$ iff Ker $A=\{0\}$. (To be
clear, both operators are considered as continuous maps of the form
$\mathcal{S}(\mathbb{R}^{m}) \longrightarrow\mathcal{S}(\mathbb{R}^{m})$,
where $m=n$ and $m=n-k$ for $\widetilde{A}$ and $A$, respectively).
\end{lemma}

\begin{proof}
Assume that Ker $A\not =\{0\}$ and let $\psi\in\mathcal{S}(\mathbb{R}%
^{n-k})\backslash\{0\}$ be such that $A \psi=0$. Let $\Psi=T_{\widetilde{S}%
,\chi}\psi=\widetilde{S}^{-1} \psi\otimes\chi$, where $\widetilde{S}$ is (one
of the two) metaplectic operators associated with $s$, and $\chi\in
\mathcal{S}(\mathbb{R}^{k})\backslash\{0\}$. Then $\Psi\not =0$ and in view of
proposition \ref{proptic1}
\[
\widetilde{A} \Psi=\widetilde{A} T_{\widetilde{S},\chi}\psi= T_{\widetilde{S}%
,\chi} A\psi=0\, .
\]
We conclude that Ker $\widetilde{A}\not =\{0\}$.

Conversely, assume that Ker $\widetilde{A}\not =\{0\}$ and let $\Psi
\in\mathcal{S}(\mathbb{R}^{n})\backslash\{0\}$ be such that $\widetilde{A}%
\Psi=0$. Once again, in view of Proposition \ref{proptic1} we have for
$\psi=T^{*}_{\widetilde{S},\chi} \Psi$
\[
A \psi= A T^{*}_{\widetilde{S},\chi} \Psi=T^{*}_{\widetilde{S},\chi}
\widetilde{A}\Psi=0.
\]
It remains to prove that there is always some $\chi\in\mathcal{S}%
(\mathbb{R}^{k})$ such that $\psi=T^{*}_{\widetilde{S},\chi} \Psi\not =0$.
From proposition \ref{propONB} we know that $(T_{\widetilde{S},\chi_{l}}
\phi_{j})_{jl}$ is an orthogonal basis of $L^{2}(\mathbb{R}^{n})$ provided
$(\phi_{j})_{j}$ and $(\chi_{l})_{l}$ are orthogonal basis of $L^{2}%
(\mathbb{R}^{n-k})$ and $L^{2}(\mathbb{R}^{k})$, respectively. In addition, we
may choose $\phi_{j} \in\mathcal{S}(\mathbb{R}^{n-k})$ and $\chi_{l}%
\in\mathcal{S}(\mathbb{R}^{k})$. It follows that there is always some $l,j$
such that
\[
(\Psi|T_{\widetilde{S},\chi_{l}} \phi_{j})=(T^{*}_{\widetilde{S},\chi_{l}}
\Psi|\phi_{j}) \not =0
\]
and so $\psi=T^{*}_{\widetilde{S},\chi_{l}} \Psi\not =0$, which concludes the proof.
\end{proof}

\begin{remark}
Every operator $A\in{HG}_{\rho}^{m_{1},m_{0}}(\mathbb{R}^{2(n-k)})$ extends to
a continuous operator of the form $A^{\prime}:\mathcal{S}^{\prime}%
(\mathbb{R}^{n-k}) \longrightarrow\mathcal{S}^{\prime}(\mathbb{R}^{n-k})$. Its
kernel, however, satisfies [Shubin \cite{sh87}, p.187]
\[
\mbox{Ker} \,\, \left(  \left.  A^{\prime}\right|  _{\mathcal{S}^{\prime
}(\mathbb{R}^{n-k})} \right)  = \mbox{Ker} \,\, \left(  \left.  A \right|
_{\mathcal{S}(\mathbb{R}^{n-k})} \right)  \, .
\]

This property is not shared by the continuous extension $\widetilde{A}%
^{\prime}:\mathcal{S}^{\prime}(\mathbb{R}^{n}) \longrightarrow\mathcal{S}%
^{\prime}(\mathbb{R}^{n})$ of $\widetilde{A}=\mathbb{E}_{s}[A]$. If Ker
$\widetilde{A}^{\prime}\not =\{0\}$ we always have
\[
(\mbox{Ker}\,\,\widetilde{A}^{\prime}) \cap(\mathcal{S}^{\prime}%
(\mathbb{R}^{n}) \backslash\mathcal{S}(\mathbb{R}^{n})) \not =\emptyset\, .
\]
We will prove this relation in proposition \ref{propGH2} where we will also
show that the operators $\widetilde{A}\in\widetilde{HG}_{\rho}^{m_{1},m_{0}%
}(\mathbb{R}^{2n})$ are not, in general, globally hypoelliptic.
\end{remark}

The spectral properties of the operators in the classes $\widetilde{HG}%
^{m_{1},m_{0}}_{\rho}(\mathbb{R}^{2n})$, for $m_{0}>0$, follow from
propositions \ref{spectraltheor} and \ref{propshubin}

\begin{proposition}
\label{propSP} Let $\widetilde{A}\in\widetilde{HG}_{\rho}^{m_{1},m_{0}%
}(\mathbb{R}^{2n})$ and let $m_{0}>0$. Assume in addition that $\widetilde{A}$
is formally self-adjoint. Then

(i) $\widetilde{A}$ has discrete spectrum $(\lambda_{j})_{j\in\mathbb{N}}$ and
$\lim_{j\rightarrow\infty}|\lambda_{j}|=\infty$;

(ii) A complete set of eigenfunctions of $\widetilde{A}$ is given by
$\Phi_{jl}=T_{\widetilde{S},\chi_{l}}\phi_{j}$ where $(\phi_{j})_{j}$ is a
complete set of eigenfunctions of the associated operator $A \in HG_{\rho
}^{m_{1},m_{0}}(\mathbb{R}^{2(n-k)})$ (such that $\widetilde{A}=\mathbb{E}%
_{s}[A]$) and $(\chi_{l} \in\mathcal{S}(\mathbb{R}^{k}))_{l}$ is an
orthonormal basis of $L^{2}(\mathbb{R}^{k})$;

(iii) We have $\Phi_{jl}\in\mathcal{S}(\mathbb{R}^{n})$ and $\left(  \Phi
_{jl}\right)  _{jl}$ form an orthonormal basis of $L^{2}(\mathbb{R}^{n})$.
\end{proposition}

\begin{proof}
From definition \ref{definitionESC}, $\widetilde{A} \in\widetilde{HG}_{\rho
}^{m_{1},m_{0}}(\mathbb{R}^{2n})$ iff there is some operator $A \in{HG}_{\rho
}^{m_{1},m_{0}}(\mathbb{R}^{2(n-k)})$ such that $\widetilde{A}=\mathbb{E}%
_{s}[A]$ for some extension map $\mathbb{E}_{s}$. Recall, also from
proposition \ref{propSA}, that $\widetilde{A}$ is formally self-adjoint iff
$A$ is also formally self-adjoint. The results of the present theorem are then
corollaries of the spectral theorem \ref{spectraltheor} after taking into
account the spectral properties of Shubin's classes listed in Proposition
\ref{propshubin} above.
\end{proof}

We also have

\begin{proposition}
\label{propIV} Let $\widetilde{A} \in\widetilde{HG}_{\rho}^{m_{1},m_{0}%
}(\mathbb{R}^{2n})$. If Ker $\widetilde{A}=$Ker $\widetilde{A}^{*}=\{0\}$ then
$\widetilde{A}$ is invertible with inverse $\widetilde{A}^{-1} \in
\widetilde{HG}_{\rho}^{-m_{0},-m_{1}}(\mathbb{R}^{2n})$.
\end{proposition}

\begin{proof}
If $\widetilde{A} \in\widetilde{HG}_{\rho}^{m_{1},m_{0}}(\mathbb{R}^{2n})$
then also $\widetilde{A}^{*} \in\widetilde{HG}_{\rho}^{m_{1},m_{0}}%
(\mathbb{R}^{2n})$ and there are some $A,A^{*}\in{HG}_{\rho}^{m_{1},m_{0}%
}(\mathbb{R}^{2(n-k)})$ such that $\widetilde{A}=\mathbb{E}_{s}[A]$ and
$\widetilde{A}^{*}=\mathbb{E}_{s}[A^{*}]$. If Ker $\widetilde{A}=$Ker
$\widetilde{A}^{*}=\{0\}$ then, in view of Lemma \ref{Kernelrelation}, we also
have Ker $A=$Ker $A^{*}=\{0\}$. It follows from Shubin \cite{sh87} that $A$ is
invertible with inverse $A^{-1} \in{HG}_{\rho}^{-m_{0},-m_{1}}(\mathbb{R}%
^{2(n-k)})$. From Corollary \ref{CorIN} we conclude that $\mathbb{E}%
_{s}[A^{-1}] \in\widetilde{HG}_{\rho}^{-m_{0},-m_{1}}(\mathbb{R}^{2n})$ is the
inverse of $\widetilde{A}$.
\end{proof}

Contrary to what happens in ${HG}_{\rho}^{m_{1},m_{0}}(\mathbb{R}^{2n})$, the
operators $\widetilde{A} \in\widetilde{HG}_{\rho}^{m_{1},m_{0}}(\mathbb{R}%
^{2n})$ are not in general globally hypoelliptic. A set of sufficient
conditions for hypoellipticity in $\widetilde{HG}_{\rho}^{m_{1},m_{0}%
}(\mathbb{R}^{2n})$ is given by

\begin{proposition}
\label{propGH1} If $\widetilde{A} \in\widetilde{HG}_{\rho}^{m_{1},m_{0}%
}(\mathbb{R}^{2n})$ and Ker $\widetilde{A}=$Ker $\widetilde{A}^{*}=\{0\}$ then
the extension $\widetilde{A}^{\prime}:\mathcal{S}^{\prime}(\mathbb{R}%
^{n})\longrightarrow\mathcal{S}^{\prime}(\mathbb{R}^{n})$ is globally hypoelliptic.
\end{proposition}

\begin{proof}
In view of the previous proposition, $\widetilde{A}:\mathcal{S}(\mathbb{R}%
^{n}) \longrightarrow\mathcal{S}(\mathbb{R}^{n})$ is invertible. Then its
extension by duality $\widetilde{A}^{\prime}:\mathcal{S}^{\prime}%
(\mathbb{R}^{n}) \longrightarrow\mathcal{S}^{\prime}(\mathbb{R}^{n})$ is also
invertible (Remark \ref{remark-duality}) and satisfies
\[
(\widetilde{A}^{\prime})^{-1}=(\widetilde{A}^{-1})^{\prime}\Longrightarrow\left.
(\widetilde{A}^{\prime})^{-1} \right|
_{\mathcal{S}(\mathbb{R}^{n})}= \widetilde{A}^{-1}\, .
\]
Hence $(\widetilde{A}^{\prime})^{-1}
\phi\in\mathcal{S}(\mathbb{R}^{n})$ for all
$\phi\in\mathcal{S}(\mathbb{R}^{n})$. It follows that
$\widetilde{A}^{\prime }:\mathcal{S}^{\prime}(\mathbb{R}^{n})
\longrightarrow\mathcal{S}^{\prime }(\mathbb{R}^{n})$ is globally
hypoelliptic. In fact,
\[
\widetilde{A}^{\prime}\psi=\phi\in\mathcal{S}(\mathbb{R}^{n})
\Longrightarrow \psi= (\widetilde{A}^{\prime})^{-1}
\phi\in\mathcal{S}(\mathbb{R}^{n}) \, .
\]

\end{proof}

However,

\begin{proposition}
\label{propGH2} Let $\widetilde{A} \in\widetilde{HG}_{\rho}^{m_{1},m_{0}%
}(\mathbb{R}^{2n})$. If Ker $\widetilde{A}^{\prime}\not =\{0\}$ then
$\widetilde{A}^{\prime}$ is not globally hypoelliptic.
\end{proposition}

\begin{proof}
If (Ker $\widetilde{A}^{\prime}) \cap(\mathcal{S}^{\prime}(\mathbb{R}^{n})
\backslash\mathcal{S}(\mathbb{R}^{n})) \not =\emptyset$ then $\widetilde{A}%
^{\prime}$ is obviously not global hypoelliptic.

Assume that Ker $\widetilde{A}^{\prime}\subset\mathcal{S}(\mathbb{R}^{n})$.
Let $A \in{HG}_{\rho}^{m_{1},m_{0}}(\mathbb{R}^{2(n-k)})$ be such that
$\widetilde{A}=\mathbb{E}_{s}[A]$. Since Ker $\widetilde{A}=$ Ker
$\widetilde{A}^{\prime}\not =\{0\}$, it follows from Lemma
\ref{Kernelrelation} that Ker $A \not =\{0\}$. Let then $\phi\in
\mathcal{S}(\mathbb{R}^{n-k})\backslash\{0\}$ be such that $A\phi=0$ and let
$\chi\in\mathcal{S}^{\prime}(\mathbb{R}^{k})\backslash\mathcal{S}%
(\mathbb{R}^{k})$. Then $\Phi= T_{\widetilde{S},\chi} \phi\notin%
\mathcal{S}(\mathbb{R}^{n})$ and, using a trivial extension of the first
intertwining relation (\ref{inter}) to the case $\chi\notin\mathcal{S}%
(\mathbb{R}^{n})$, we get
\[
\widetilde{A}^{\prime}\Phi=\widetilde{A}^{\prime}T_{\widetilde{S},\chi} \phi=
T_{\widetilde{S},\chi} A^{\prime}\phi=0 \in\mathcal{S}(\mathbb{R}^{n})\, .
\]
Hence, Ker $\widetilde{A}^{\prime}\not \subset \mathcal{S}(\mathbb{R}^{n})$
and $\widetilde{A}^{\prime}$ is not globally hypoelliptic.
\end{proof}

\section{Examples\label{secExa}}

In this section we present a few examples of operators that belong to the new
classes $\widetilde{HG}_{\rho}^{m_{1},m_{0}}(\mathbb{R}^{2n})$ and which play
an important role in mathematical physics. We also determine the intertwining
operators that generate the complete sets of eigenfunctions for each of these
operators. Our results recover the spectral properties that were obtained in
some previous studies \cite{CPDE,Bopp1,JPDOA}.

\subsection{The Landau Hamiltonian\label{secLandau}}

The Landau Hamiltonian \cite{lali97}
\[
\widetilde{H}_{L}=-(\partial_{x}^{2}+\partial_{y}^{2})+i(x\partial
_{y}-y\partial_{x})+\tfrac{1}{4}(x^{2}+y^{2})
\]
describes the motion of a test particle in the presence of a magnetic field.
The spectral properties of this operator have been studied in detail in
\cite{CPDE,JPDOA}. Here we show that $\widetilde{H}_{L} \in\widetilde{HG}%
_{1}^{2,2}(\mathbb{R}^{4})$ and so not only its spectral properties, but also
its invertibility and hypoellipticity properties, follow directly from
propositions \ref{propSP},\ref{propIV} and \ref{propGH1}.

\begin{proposition}
\label{propL1} We have $\widetilde{H}_{L} \in\widetilde{HG}_{1}^{2,2}%
(\mathbb{R}^{4})$.
\end{proposition}

\begin{proof}
Recall that the harmonic oscillator Hamiltonian
\[
H_{0}=-\partial_{x}^{2}+x^{2}%
\]
is a Weyl operator with symbol $h_{0}(x,\xi)=\xi^{2}+x^{2}$. We have, of
course, $h_{0}\in{H\Gamma}_{1}^{2,2}(\mathbb{R}^{2})$. The Weyl symbol of the
Landau Hamiltonian is
\[
\widetilde{h}_{L}(x,y;\xi,\eta)=\xi^{2}+\eta^{2}-x\eta+y\xi+\frac{1}{4}%
(x^{2}+y^{2})
\]
and one can easily check that $\widetilde{h}_{L}=E_{s}[h_{0}]$
where
$E_{s}:\mathcal{S}^{\prime}(\mathbb{R}^{2})\longrightarrow\mathcal{S}^{\prime
}(\mathbb{R}^{4})$ is the symplectic dimensional extension map
associated with $s\in$ $\operatorname*{Sp}(4,\mathbb{R})$
\begin{equation}
{s}=%
\begin{pmatrix}
\frac{1}{2}I & -D\\
\frac{1}{2}D & I
\end{pmatrix}
\text{ ,}\ D=%
\begin{pmatrix}
0 & 1\\
1 & 0
\end{pmatrix}
.\label{sL}%
\end{equation}
Hence, $\widetilde{h}_{L}\in\widetilde{H\Gamma}_{1}^{2,2}(\mathbb{R}^{4})$ and
so $\widetilde{H}_{L}\in\widetilde{HG}_{1}^{2,2}(\mathbb{R}^{4})$.
\end{proof}

Since Ker $H_{0}=\{0\}$, we conclude from Lemma \ref{Kernelrelation} that also
Ker $\widetilde{H}_{L}=\{0\}$. Since $\widetilde{H}_{L}$ is formally
self-adjoint, it is also globally hypoelliptic and invertible with inverse
$\widetilde{H}_{L}^{-1} \in\widetilde{HG}_{1}^{-2,-2}(\mathbb{R}^{4})$
(Propositions \ref{propIV} and \ref{propGH1}).

Moreover, $\widetilde{H}_{L}$ displays a discrete spectrum (Proposition
\ref{propSP}). The intertwining operators (\ref{tics}) that generate the
eigenfunctions of $\widetilde{H}_{L}$ from those of $H_{0}$ are given by

\begin{proposition}
\label{propL2} The intertwiners for the Landau Hamiltonian are explicitly
\begin{equation}
\label{TL}T_{\widetilde{S},\chi}\phi(x,y)=-i\left(  \frac{\pi}{2}\right)
^{1/2}W(\phi,\overline{\hat{\chi}})\left(  \tfrac{x}{2},\tfrac{y}{2}\right)
\end{equation}
where $W$ is the cross Wigner function
\begin{equation}
\label{CWF}W(\phi,\psi)(x,y)=\left(  \frac{1}{2\pi}\right)  ^{n}%
\int_{\mathbb{R}^{n}}e^{-iy\cdot x^{\prime}}\phi(x+\tfrac{1}{2}x^{\prime
})\overline{\psi(x-\tfrac{1}{2}x^{\prime})}\, dx^{\prime}%
\end{equation}
for $n=1$ and $\hat\chi$ is the Fourier transform of $\chi$.
\end{proposition}

\begin{proof}
We have from (\ref{tics}) that
\[
T_{\widetilde{S},\chi}\phi=\widetilde{S}^{-1}(\phi\otimes\chi)
\]
where $\widetilde{S}^{-1}$ is one of the two metaplectic operators that
projects into the symplectomorphism which is the inverse of $s\in$
$\operatorname*{Sp}(4,\mathbb{R})$ given by (\ref{sL})
\[
{s}^{-1}=%
\begin{pmatrix}
I & D\\
-\frac{1}{2}D & \frac{1}{2}I
\end{pmatrix}
\text{ , }\ D=%
\begin{pmatrix}
0 & 1\\
1 & 0
\end{pmatrix}
.
\]
The quadratic form $W(x,y;x^{\prime},y^{\prime})$ that generates ${s}^{-1}$
satisfies (cf. formula (\ref{wgen}))
\begin{equation}
(x,y;\xi,\eta)=s^{-1}(x^{\prime},y^{\prime};\xi^{\prime},\eta^{\prime
})\Longleftrightarrow\left\{
\begin{array}
[c]{c}%
(\xi,\eta)=(\partial_{x}W,\partial_{y}W)\\
(\xi^{\prime},\eta^{\prime})=-(\partial_{x^{\prime}}W,\partial_{y^{\prime}}W)
\end{array}
\right.  \,\label{defW}%
\end{equation}
We have
\[
W(x,y;x^{\prime},y^{\prime})=-(y\cdot x^{\prime}+x\cdot y^{\prime})+x^{\prime
}\cdot y^{\prime}+\tfrac{1}{2}x\cdot y\,.
\]
Choosing the Maslov index $m=0$ in (\ref{deltaw}) and noticing that $|\det$
$L|=1$ we get from (\ref{swm}):
\begin{equation}
\widetilde{S}^{-1}\Psi(x,y)=\left(  \tfrac{1}{2\pi i}\right)  \int%
_{\mathbb{R}^{2}}e^{iW(x,y;x^{\prime},y^{\prime})}\Psi(x^{\prime},y^{\prime
})\,dx^{\prime}dy^{\prime}.\label{stipsi}%
\end{equation}
Setting $\Psi(x,y)=\phi(x)\overline{\hat{\chi}(y)}$ and using
\begin{equation}
\overline{\hat{\chi}(y)}=\left(  \tfrac{1}{2\pi}\right)  ^{n/2}\int%
_{\mathbb{R}^{n}}e^{iz\cdot y}\overline{\chi(z)}\,dz\label{kihat}%
\end{equation}
with $n=1$, we get
\[
T_{\widetilde{S},\overline{\hat{\chi}}}\phi(x,y)=-i\left(  \tfrac{1}{2\pi
}\right)  ^{3/2}\int_{\mathbb{R}^{3}}e^{iW(x,y;x^{\prime},y^{\prime}%
)}e^{izy^{\prime}}\phi(x^{\prime})\overline{\chi(z)}\,dzdx^{\prime}dy^{\prime
}\,.
\]
A trivial calculation then shows that
\[
\widetilde{S}^{-1}\phi\otimes\overline{\hat{\chi}}(x,y)=\left(  \tfrac{-i}%
{2}\right)  \left(  \tfrac{1}{2\pi}\right)  ^{1/2}\int_{\mathbb{R}}%
e^{-i\tfrac{y}{2}\cdot x^{\prime}}\phi(\tfrac{1}{2}(x+x^{\prime}%
))\overline{\chi(\tfrac{1}{2}(x-x^{\prime}))}\,dx^{\prime}%
\]
and so
\[
T_{\widetilde{S},\overline{\hat{\chi}}}\phi(x,y)=-i\left(  \frac{\pi}%
{2}\right)  ^{1/2}W(\phi,\chi)\left(  \tfrac{x}{2},\tfrac{y}{2}\right)  .
\]
Taking into account that $\varphi=\overline{\hat{\chi}}\Longrightarrow
\chi=\overline{\hat{\varphi}}$ we easily obtain (\ref{TL}).
\end{proof}

This result coincides exactly with the formula for the intertwiners that was
obtained in \cite{CPDE,JPDOA}.

\subsection{Bopp pseudo-differential operators}

\label{secBopp}

Let $a \in\mathcal{S}^{\prime}(\mathbb{R}^{2n})$ and consider the operator
defined by
\begin{equation}
\widetilde{A}_{B}:\mathcal{S}(\mathbb{R}^{2n}) \longrightarrow\mathcal{S}%
^{\prime}(\mathbb{R}^{2n});\,\, \Psi\longrightarrow\widetilde{A}_{B}\Psi=a
\star\Psi
\end{equation}
where $\star$ is the Moyal star-product
\begin{equation}
a \star b (z)=\left(  \frac{1}{4\pi} \right)  ^{2n} \int_{\mathbb{R}%
^{2n}\times\mathbb{R}^{2n}}e^{-\tfrac{i}{2}\sigma_{n}(v,u)} a\left(
z+\tfrac{1}{2}u\right)  b\left(  z-\tfrac{1}{2}v\right)  dudv
\end{equation}
familiar from deformation quantization \cite{BFFLS1,BFFLS2,DiPa1,DiPa2}. The
operators $\widetilde{A}_{B}=a \star$ are usually called "Bopp
pseudo-differential operators" \cite{bopp} (or "Moyal-Weyl pseudo-differential
operators" \cite{DiGoLuPa3}) and are very important in deformation
quantization where they were used, for instance, to obtain a precise
formulation of the spectral problem \cite{Bopp1}.

Moreover, the map $a \longrightarrow\widetilde{A}_{B}=a\star$ yields a phase
space representation of quantum mechanics \cite{DiGoLuPa3}, called the "Moyal
representation", that finds interesting applications in semiclassical dynamics
\cite{Bondar1}, hybrid dynamics \cite{Bondar2} and non-commutative quantum
mechanics \cite{BaBeDiPa,BaDiPa,DiGoLuPa2,DiGoLuPa1}.

In this section we show that for every $a \in{H\Gamma}_{\rho}^{m_{1},m_{0}%
}(\mathbb{R}^{2n})$, the operator $\widetilde{A}_{B}=a \star$ belongs to
$\widetilde{HG}_{\rho}^{m_{1},m_{0}}(\mathbb{R}^{4n})$. Hence its spectral,
invertibility and hypoellipticity properties are given by our Propositions
\ref{propSP},\ref{propIV},\ref{propGH1} and \ref{propGH2}.

We start by recalling that

\begin{lemma}
Let $a\in\mathcal{S}^{\prime}(\mathbb{R}^{2n})$. Then the operator formally
$\widetilde{A}_{B}=a \star$ is a Weyl operator $\mathcal{S}(\mathbb{R}^{2n})
\longrightarrow\mathcal{S}^{\prime}(\mathbb{R}^{2n})$ with symbol
\begin{equation}
\label{Bopp-Weyl}\widetilde{a}_{B}(x,y;\xi,\eta)=a(x-\frac{1}{2}\eta
,y+\frac{1}{2}\xi).
\end{equation}

\end{lemma}

For a proof see \cite{Bopp1}. We then have

\begin{proposition}
Let $a \in{H\Gamma}_{\rho}^{m_{1},m_{0}}(\mathbb{R}^{2n})$. Then
$\widetilde{A}_{B} \in\widetilde{HG}_{\rho}^{m_{1},m_{0}}(\mathbb{R}^{4n})$.
\end{proposition}

\begin{proof}
The Weyl symbol $\widetilde{a}_{B}$ (\ref{Bopp-Weyl}) is given by
$\widetilde{a}_{B}=E_{s}[a]$ where $E_{s}:\mathcal{S}^{\prime}(\mathbb{R}%
^{2n})\longrightarrow\mathcal{S}^{\prime}(\mathbb{R}^{4n})$ is the
symplectic dimensional
extension map associated with the symplectomorphism $s\in$ $\operatorname*{Sp}%
(4n,\mathbb{R})$
\begin{equation}
{s}=%
\begin{pmatrix}
I & -\frac{1}{2}D\\
D & \frac{1}{2}I
\end{pmatrix}
\text{ , }\ D=%
\begin{pmatrix}
0 & I\\
I & 0
\end{pmatrix}
.\label{sym-Bopp}%
\end{equation}
Hence, for $a\in{H\Gamma}_{\rho}^{m_{1},m_{0}}(\mathbb{R}^{2n})$ we have
$\widetilde{a}_{B}\in\widetilde{H\Gamma}_{\rho}^{m_{1},m_{0}}(\mathbb{R}%
^{4n})$ and so $\widetilde{A}_{B}\in\widetilde{HG}_{\rho}^{m_{1},m_{0}%
}(\mathbb{R}^{4n})$.
\end{proof}

Let us then determine the explicit form of the intertwining operators that
generate the eigenfunctions of $\widetilde{A}_{B}=a\star$ from those of $A=$
Op$^{w}(a)$.

\begin{proposition}
The intertwiners for the Bopp operators are given by
\begin{equation}
T_{\widetilde{S},\chi}\phi=\left(  2\pi\right)  ^{n/2}i^{-n}W(\phi
,\overline{\hat{\chi}}) \label{tw}%
\end{equation}
where $W$ is the cross-Wigner function (\ref{CWF}).
\end{proposition}

\begin{proof}
The intertwining operators are given by $T_{\widetilde{S},\chi}\phi
=\widetilde{S}^{-1}\phi\otimes\chi$ where $\widetilde{S}^{-1}$ is (one of the
two) metaplectic operators that project into the symplectomorphism $s^{-1}\in$
$\operatorname*{Sp}(4n,\mathbb{R})$, inverse of (\ref{sym-Bopp}).
\[
{s}^{-1}=%
\begin{pmatrix}
\frac{1}{2}I & \frac{1}{2}D\\
-D & I
\end{pmatrix}
\text{ ,}\ D=%
\begin{pmatrix}
0 & I\\
I & 0
\end{pmatrix}
.
\]
From (\ref{defW}) we find the quadratic form that generates $s^{-1}$
\[
W(x,y;x^{\prime},y^{\prime})=-2(y\cdot x^{\prime}+x\cdot y^{\prime}%
)+x^{\prime}\cdot y^{\prime}+2x\cdot y.
\]
Choosing the Maslov index $m=0$ in (\ref{deltaw}) and taking into account that
$|\det$ $L|=2^{2n}$ we get from (\ref{swm})
\begin{equation}
\widetilde{S}^{-1}\Psi(x,y)=\left(  \tfrac{1}{\pi i}\right)  ^{n}%
\int_{\mathbb{R}^{2n}}e^{iW(x,y;x^{\prime},y^{\prime})}\Psi(x^{\prime
},y^{\prime})\,dx^{\prime}dy^{\prime}.\label{stipsi2}%
\end{equation}
Setting $\Psi=\phi\otimes\overline{\hat{\chi}}$ and using (\ref{kihat}) we
get
\begin{align*}
T_{\widetilde{S},\overline{\hat{\chi}}}\phi(x,y) &  =\left(  \tfrac{1}{2\pi
}\right)  ^{n/2}\left(  \tfrac{1}{i\pi}\right)  ^{n}\int_{\mathbb{R}^{3n}%
}e^{iW(x,y;x^{\prime},y^{\prime})}e^{iz\cdot y^{\prime}}\phi(x^{\prime
})\overline{\chi(z)}\,dzdx^{\prime}dy^{\prime}\\
&  =\left(  \tfrac{2}{\pi}\right)  ^{n/2}i^{-n}\int_{\mathbb{R}^{n}%
}e^{-iy\cdot(2x^{\prime}-2x)}\phi(x^{\prime})\overline{\chi(2x-x^{\prime}%
)}\,dx^{\prime}.
\end{align*}
Making the change of variables $x^{\prime\prime}=2x^{\prime}-2x$ and taking
into account that $\varphi=\overline{\hat{\chi}}\Longrightarrow\chi
=\overline{\hat{\varphi}}$ we easily obtain (\ref{tw}).
\end{proof}

{\vskip 1.0cm}

\noindent\textbf{Acknowledgements}. The authors would like to
thank the anonymous referee for making an important correction in
section 5.

Nuno Costa Dias and Jo\~{a}o Nuno Prata have been supported by the
research grant PTDC/MAT/099880/2008 of the Portuguese Science
Foundation (FCT). Maurice de Gosson has been supported by a
research grant from the Austrian Research Agency FWF
(Projektnummer P23902-N13).

\pagebreak

**********************************************************************************************************************************************************************************************************

\textbf{Author's addresses:}

\begin{itemize}
\item \textbf{Jo\~{a}o Nuno Prata }and\textbf{\ Nuno Costa Dias:
}Departamento de Matem\'{a}tica. Universidade Lus\'{o}fona de
Humanidades e Tecnologias. Av. Campo Grande, 376, 1749-024 Lisboa,
Portugal

and

Grupo de F\'{\i}sica Matem\'{a}tica, Universidade de Lisboa, Av.
Prof. Gama Pinto 2, 1649-003 Lisboa, Portugal

\item \textbf{Maurice A. de Gosson:} Universit\"{a}t Wien,
Fakult\"{a}t f\"{u}r Mathematik--NuHAG, Nordbergstrasse 15, 1090
Vienna, Austria
\end{itemize}

**********************************************************************************************************************************************************************************************************

\begin{thebibliography}{99}                                                                                               %


\bibitem {BaBeDiPa}C. Bastos, O. Bertolami, N.C. Dias, and J.N. Prata,
Weyl--Wigner Formulation of Noncommutative Quantum Mechanics, J. Math. Phys.
49 (2008) 072101 (24 pages).

\bibitem {BaDiPa}C. Bastos, N.C. Dias, and J.N. Prata, Wigner measures in
noncommutative quantum mechanics, Comm. Math. Phys. 299 (2010), no.3, 709--740.

\bibitem {BFFLS1}F. Bayen, M. Flato, C. Fronsdal, A.\ Lichnerowicz, and
D.\ Sternheimer, {Deformation Theory and Quantization. I. Deformation of
Symplectic Structures}, Annals of Physics, \textbf{110} (1978), 111--151.

\bibitem {BFFLS2}F. Bayen, M. Flato, C. Fronsdal, A.\ Lichnerowicz, and
D.\ Sternheimer, {Deformation Theory and Quantization. II. Physical
Applications}, Annals of Physics, \textbf{111 }(1978), 6--110.

\bibitem {boburo96}P. Boggiatto, E. Buzano, and L. Rodino, Global
Hypoellipticity and Spectral Theory, Akademie Verlag, Math. Research,
\textbf{92}(1996)

\bibitem {Bondar1}D.I. Bondar, R. Cabrera and H.A. Rabitz, Wigner function's
negativity demystified. Arxiv: 1202.3628 (2012).

\bibitem {Bondar2}D.I. Bondar, R. Cabrera, R.R. Lompay, M.Y. Ivanov and H.A.
Rabitz, Operational dynamical modeling transcending quantum and classical
mechanics, Phys. Rev. Lett. 109, 190403 (2012).

\bibitem {bopp}F. Bopp, La m\'{e}canique quantique est-elle une m\'{e}canique
statistique particuli\`{e}re?\ Ann. Inst. H. Poincar\'{e}, 15 (1956), 81--112

\bibitem {DiGoLuPa2}N.C. Dias, M. de Gosson, F. Luef, J.N. Prata, A
Deformation Quantization Theory for Non-Commutative Quantum Mechanics, J.
Math. Phys. 51 (2010) 072101 (12 pages).

\bibitem {DiGoLuPa1}N.C. Dias, M. de Gosson, F. Luef, J.N. Prata, A
pseudo--differential calculus on non--standard symplectic space; spectral and
regularity results in modulation spaces. J. Math. Pure Appl. 96 (2011) 423--445

\bibitem {DiGoLuPa3}N.C. Dias, M. de Gosson, F. Luef, J.N. Prata, Quantum
mechanics in phase space: the Schr\"odinger and the Moyal representations.
Journal of Pseudo-Differential Operators and Applications, 3(4) (2012) 367-398.

\bibitem {DiPa1}N.C. Dias, J.N. Prata, Formal solutions of stargenvalue
equations, Ann. Phys.(N. Y.), 311 (2004) 120--151.

\bibitem {DiPa2}N.C. Dias, J.N. Prata, Admissible states in quantum phase
space, Ann. Phys.(N. Y.), 313 (2004) 110--146.

\bibitem {Folland}G. B. Folland. Harmonic Analysis in Phase space, Annals of
Mathematics studies, Princeton University Press, Princeton, N.J. (1989)

\bibitem {wiley}M. de Gosson. Maslov Classes, Metaplectic Representation and
Lagrangian Quantization, Research Notes in Mathematics, 95, Wiley--VCH,
Berlin, (1997)

\bibitem {Birk}M. de Gosson. Symplectic Geometry and Quantum Mechanics,
Birkh\"{a}user, Basel, series \textquotedblleft Operator Theory: Advances and
Applications\textquotedblright\ (subseries: \textquotedblleft Advances in
Partial Differential Equations\textquotedblright), Vol. 166 (2006)

\bibitem {Birkbis}M. de Gosson. Symplectic Methods in Harmonic Analysis and in
Mathematical Physics. Birkh\"{a}user; Springer Basel (2011)

\bibitem {CPDE}M. de Gosson. Spectral Properties of a Class of Generalized
Landau Operators, Comm. Partial Differential Equations, 33(11) (2008), 2096--2104

\bibitem {Bopp1}M. de Gosson and F. Luef.\ A new approach to the $\star
$-genvalue equation, Lett. Math. Phys.\ 85 (2008), 173--183

\bibitem {JPDOA}M. de Gosson and F. Luef. Spectral and Regularity properties
of a Pseudo-Differential Calculus Related to Landau Quantization, Journal
of\ Pseudo-Differential Operators and Applications, 1(1) (2010), 3--34

\bibitem {lali97}L. D. Landau and E. M. Lifshitz. Quantum Mechanics:
Nonrelativistic Theory, Pergamon Press (1997)

\bibitem {Leray}J. Leray. Lagrangian Analysis and Quantum Mechanics,\ a
mathematical structure related to asymptotic expansions and the Maslov index,
MIT Press, Cambridge, Mass. (1981)

\bibitem {sh87}M. A. Shubin Pseudodifferential Operators and Spectral Theory,
Springer-Verlag (1987) [original Russian edition in Nauka, Moskva (1978)]

\bibitem {Wong}M. W. Wong. Weyl Transforms\textit{.} Springer, 1998\bigskip
\end{thebibliography}
\end{document}